\documentclass{amsart}
\usepackage{amsmath,amssymb, amscd, stmaryrd}
\usepackage{fullpage}
\usepackage{enumerate}
\usepackage{cite}
\usepackage{framed}
\usepackage[dvips,dvipdf]{graphicx}
\usepackage[usenames,dvipsnames]{color}
\usepackage{color}
\usepackage{tabularx}
\usepackage{array}
\usepackage{multirow}
\usepackage{changes}
\newtheorem{theorem}{Theorem}[section]
\newtheorem{proposition}[theorem]{Proposition}

\theoremstyle{definition}

\theoremstyle{definition}
\newtheorem{algorithm}[theorem]{Algorithm}
\theoremstyle{remark}
\newtheorem{remark}[theorem]{Remark}
\newtheorem{example}[theorem]{Example}

\newcommand{\mb}[1]{\ensuremath{\mathbf{#1}}}
\newcommand{\ip}[2]{\ensuremath{\langle #1,#2\rangle}}

\newcommand{\CC}{\mathbb C}
\newcommand{\NN}{\mathbb N}
\newcommand{\PP}{\mathbb P}

\newcommand{\RR}{\mathbb R}

\newcommand{\cO}{\mathcal O}
\newcommand{\cT}{\mathcal T}
\newcommand{\cH}{\mathcal H}
\newcommand{\cI}{\mathcal I}
\newcommand{\cV}{\mathcal V}
\newcommand{\cE}{\mathcal E}

\newcommand{\argmax}{\mathrm{arg\,max}}

\author[A. Anand]{Akash Anand} 
\address{Akash Anand, Department of
  Mathematics and Statistics, Indian Institute of Technology, Kanpur, UP 208016}
\email{akasha@iitk.ac.in}

\author[J. Ovall]{Jeffrey S. Ovall} \address{Jeffrey S. Ovall,
  Fariborz Maseeh Department of Mathematics and Statistics, Portland
  State University, Portland, OR 97201}
\email{jovall@pdx.edu}

\author[S. Reynolds]{Samuel E. Reynolds} \address{Samuel Reynolds, Fariborz Maseeh Department of Mathematics and Statistics, Portland
  State University, Portland, OR 97201}
\email{ser6@pdx.edu}

\author[S. Wei\ss er]{Steffen Wei\ss er} \address{Steffen Wei\ss er, Department of Mathematics,
  Saarland University, 66041 Saarbr\"ucken, Germany}
\email{weisser@num.uni-sb.de}

\thanks{This work was partially supported by the NSF Grants
  DMS-1522471 and DMS-1624776.  Numerical studies were facilitated 
  by the Portland Institute for Computational Sciences.}

\begin{document}
\title[Trefftz Finite Elements on Curvilinear Polygons]{Trefftz Finite
  Elements on Curvilinear Polygons} \date{\today}

\begin{abstract}
  We present a Trefftz-type finite element method on meshes consisting
  of curvilinear polygons.  Local basis functions are computed using
  integral equation techniques that allow for the efficient and accurate
  evaluation of quantities needed in the formation of local stiffness
  matrices.  To define our local finite element spaces in the presence
  of curved edges, we must also properly define what it means for a
  function defined on a curved edge to be ``polynomial'' of a given
  degree on that edge.  We consider two natural choices, before
  settling on the one that yields the inclusion of complete polynomial
  spaces in our local finite element spaces, and discuss how to work
  with these edge polynomial spaces in practice.  An interpolation
  operator is introduced for the resulting finite elements, and we
  prove that it provides optimal order convergence for interpolation
  error under reasonable assumptions.  We provide a description of the
  integral equation approach used for the examples in this paper,
  which was recently developed precisely with these applications in
  mind.  A few numerical examples illustrate this optimal order
  convergence of the finite element solution on some families of meshes
  in which every element has at least one curved edge.  We also
  demonstrate that it is possible to exploit the approximation power
  of locally singular functions that may exist in our finite element
  spaces in order to achieve optimal order convergence without the
  typical adaptive refinement toward singular points.
\end{abstract}

\maketitle



\section{Introduction}\label{Intro}
Polygonal and polyhedral meshes in finite element analysis for the
numerical treatment of boundary value problems have attracted a lot of
interest during the last few years due to their enormous
flexibility. They resolve the paradigm of a small class of element
shapes (e.g. triangles, quadrilaterals, tetrahedra, etc.) in finite
element methods (FEM) and therefore open the possibility for very
problem adapted mesh handling. This comes with an easy realization of
local mesh refinement, coarsening and adaptation near singularities
and interfaces. In particular, the notion of such general meshes
naturally deal with ``hanging nodes''---allowing two edges of a
polygon to meet at a straight angle removes the notion of hanging
nodes altogether.  Virtual Element Methods (VEM)
(cf.~\cite{vemMMMAS2013,vemCMA2013,vemIMA2014,vemMMMAS2014,vemIMAJNA2014,vemSINUM2014,vemCMAME2014,MR3509090,MR3564679,ANTONIETTI2017,AntoniettiBerroneVeraniWeisser2019,VeigaRussoVacca2019}),
which have drawn inspiration from mimetic finite difference schemes,
constitute one active line of research in this direction.  Another
involves Boundary Element-Based Finite Element Methods (BEM-FEM)
(cf.~\cite{bemfemLNCSE2009,bemfemETNA2010,bemfemJNM2011,bemfemNM2011,bemfemSINUM2012,bemfemJCAM2014,bemfemCMA2014,HofreitherLangerWeisser2016,Weisser2017,bemfemCMAM2018,Weisser2019,Weisser2019Book}),
which have looked more toward the older Trefftz methods for
motivation. A similar strategy has been followed in our previous
work~\cite{Anand2018}, where a Nystr\"om approximation is applied for
the treatment of local boundary integral equations instead of a
boundary element method. The gained insights and flexibilities in that
work build the basis of the development in this paper.  A third line
of research involves generalized barycentric coordinates
(cf.~\cite{GRB2014,FGS2013,RGB2011a,GRB2010,NRS2014,Spring2014} and
the references in~\cite{Floater2015}), that mimic certain key
properties of standard barycentric coordinates over general element
shapes. The before mentioned approaches yield globally-conforming
discretizations, which is challenging on general meshes. However,
there has also been significant interest in various non-conforming
methods for polyhedral meshes.  We mention Compatible Discrete
Operator (CDO), Hybrid High-Order (HHO) schemes
(cf.~\cite{Bonelle2014,Bonelle2015a,Bonelle2015b,DiPietro2014,DiPietro2015a,DiPietro2015b,BottiDiPietro2018})
and Weak Galerkin (WG) schemes
(cf.~\cite{WangYe2013,WangYe2014,WangWang2014,MuWangYe2015c,MuWangYe2015b,MuWangYe2015,WangYe2016}),
as well as the recent adaptations of the discontinuous Petrov-Galerkin
method (cf.~\cite{Astaneh2018}), in this regard.

The present work builds upon~\cite{Anand2018} for second-order,
linear, elliptic boundary value problems posed on possibly curved
domains $\Omega\subset\RR^2$: Find $u\in\cH$ such that
\begin{align}\label{ModelProblem}
\int_\Omega A\nabla u\cdot\nabla v+(\mb{b}\cdot\nabla u +cu)v\,dx=\int_\Omega f v\,dx+\int_{\partial\Omega_N}gv\,ds\mbox{
for all }v\in\cH~,
\end{align}
where $\cH$ is some appropriate subspace of $H^1(\Omega)$
incorporating homogeneous Dirichlet boundary conditions, and standard
assumptions on the data $A,\mb{b},c,f,g$ ensure that the problem is
well-posed.  Although polygonal meshes are quite flexible and have
been studied intensively in recent years, there are relatively few
results in this direction that allow for curved elements in the spirit
of polygonal meshes, despite their natural appeal in fitting curved
domain boundaries and interfaces.  Early efforts at treating curved
boundaries in the finite element context, such as isoparametric
elements (cf.~\cite{Scott1973,Lenoir1986,Bernardi1989}), involve
(local) mappings of standard mesh cells to fit curved boundaries, and
these methods remain popular today.  More recently, isogeometric
analysis (cf.~\cite{CottrellHughesBazilevs2009}), which integrates the
use of splines both for modeling complex (curved) geometries and in
constructing finite elements on the resulting meshes.  This remains an
active area of research.  Two recent contributions employing
non-conforming methods over curved polygonal elements are described
in~\cite{Brezzi2006,BottiDiPietro2018}.  In terms of conforming
methods for treating curved boundaries that are in the same vein as
the conforming polygonal methods mentioned in the first paragraph, we
mention four, all of which are very recent.
In~\cite{BertoluzzaPennacchioPrada2019}, the curved boundary of the
domain is approximated by polygonal elements with straight edges and a
stabilization is constructed such that optimal rates of convergence
are retained for high order methods. In contrast,
\cite{VeigaRussoVacca2019} gives a first study of VEM with polygonal
elements having curved edges in 2D for the treatment of curved
boundaries and interfaces, but the construction results in
$\PP_p(K)\not\subset V_p(K)$, i.e. the polynomials of degree smaller
or equal~$p$ are locally not contained in the local approximation
space of order~$p$. This introduces additional difficulties in the
study of approximation properties.  More recently, these authors work
with a richer finite element space of functions for which
$\PP_p(K)\subset V_p(K)$, see~\cite{veiga2019virtual}.  This richer
space is referred to as ``Type 2 elements'' in our previous
contribution~\cite{Anand2018}, which considered the natural
incorporation of Dirichlet data on curved (or straight) portions.  Our
present work provides a practical realization, as well as supporting
interpolation theory, for Type 2 elements on very general planar
meshes consisting of curvilinear polygons.  As both our work
and~\cite{veiga2019virtual} must address many of the same theoretical
and practical concerns, it is unsurprising that there are strong
similarities between the approaches, and comparisons between them will
be of interest as both are developed further.

The paper is organized as follows: In Section~\ref{Spaces} we describe
local and global finite element spaces allowing for mesh cells that
are fairly general curvilinear polygons.  As is done in VEM and
BEM-FEM, as well as our previous Trefftz-Nystr\"om
contribution~\cite{Anand2018}, the local spaces are defined in terms
of Poisson problems with polynomial data on the mesh cells.  For
curved edges, we discuss two natural choices (those suggested
in~\cite{Anand2018}) for what it means to have polynomial boundary
data on curved edges.  The one that we believe is the more appropriate
of the two requires further explanation concerning how these edge
polynomial spaces and their bases can be constructed in practice, and
the bulk of Section~\ref{Spaces} is devoted to doing so.  Having
defined the local and global spaces, Section~\ref{Interpolation}
provides an interpolation operator, and establishes that interpolation
in these spaces is at least as good as interpolation by polynomials in
more standard (e.g. triangular, quadrilateral) meshes, as well as
interpolation in straight-edged polygonal meshes.  In brief, it is
established that the inclusion of $\PP_p(K)$ in our local spaces
$V_p(K)$ provides the expected approximation power, and the presence
of other (possibly singular) functions in $V_p(K)$ is not
detrimental. In~\cite{Anand2018}, an example illustrated that such
locally singular functions can even be beneficial for approximation,
and we develop that argument further in the final example of
Section~\ref{Experiments}.  Section~\ref{Computation} provides a
description of the integral equation approach we use for computing the
information about our basis functions that is needed in forming finite
element stiffness matrices.  This approach, which was developed with
our present application in mind, is discussed in detail in our
previous work~\cite{Ovall2018}, so we here provide a broader
description of the approach and some of its key features.  Finally,
Section~\ref{Experiments} provides several examples illustrating the
convergence of the finite element solution on different families of
meshes whose elements each have at least one curved edge, including a
comparison of convergence and conditioning for Type 1 and Type 2
elements on families of meshes whose curved edges are very close to
being straight.  As mentioned above, in the final example of this
section we both argue and demonstrate that it is possible to exploit
the approximation power of locally singular functions that may exist
in our finite element spaces in order to achieve optimal order
convergence without the typical adaptive refinement toward singular
points.

\section{Local and Global Spaces}\label{Spaces}
Following~\cite{CiarletBook,Ern2004}, let $K$ be a connected subset of
$\RR^2$, with non-empty interior and compact closure, whose Lipschitz
boundary, $\partial K$, is a simple closed contour consisting of a
finite union of smooth arcs, see Figure~\ref{ThreePolygons}.  We will
refer to $K$ as a \textit{mesh cell}, the arcs as \textit{edges}, and
will implicitly assume that adjacent edges meet at an (interior) angle
strictly between $0$ and $2\pi$, i.e. $K$ has no slits or cusps.  We
allow adjacent edges to meet at a straight angle.  The
\textit{vertices} of $K$ are those points where two adjacent edges
meet.  Given an integer $p$ and a mesh cell $K$, we define the space
$\PP_p(K)$ to be the polynomials of (total) degree at most $p$ on $K$,
with $\PP_p(K)=\{0\}$ for $p<0$, and the space $\PP_p(\partial K)$ to
be continuous functions on $\partial K$ whose trace on each edge $e$
is the trace of a function from $\PP_p(K)$ (equivalently, from
$\PP_p(\RR^2)$) on $e$, and we denote by $\PP_p(e)$ this edge trace
space.  In~\cite{Anand2018}, we refer to this definition of $\PP_p(e)$
as its \textit{Type 2} version; the \textit{Type 1} version consists
of functions on $e$ that are polynomials with respect to a natural
parameter, such as arc length, in a parametrization of $e$.
In order to avoid unnecessary complications in our description, we
will assume that no edge is a closed contour, i.e. each edge has two
distinct endpoints, and that $K$ is simply-connected.  This is not a
necessary constraint in practice, but allowing for even more general
elements, such as those having no vertices, or those that are not
simply-connected (i.e. have holes) requires using different integral
equation techniques.  We briefly highlight this issue in
Section~\ref{Computation}.

Let $\Omega\subset\RR^2$ be a bounded domain with Lipschitz boundary.
Given a partition $\cT=\{K\}$ of $\Omega$, we define $V_p(\cT)$ by
\begin{align}\label{GlobalSpace}
V_p(\cT)=\{v\in C(\overline\Omega):\,v_{\vert_{K}}\in V_p(K)\mbox{ for
  all }K\in\cT\}~,
\end{align}
where we define the space $V_p(K)$ as follows,
\begin{align}\label{LocalSpace}
v\in V_p(K) \mbox{ if and only if }\Delta v\in\PP_{p-2}(K)\mbox{ in
  }K\mbox{ and } v\vert_{\partial K}\in \PP_p(\partial K)~.
\end{align}
The space $V_p(K)$ clearly contains $\PP_p(K)$, but it
typically contains other functions as well.
A natural decomposition of $V_p(K)$ is $V_p(K)=V_p^{K}(K)\oplus
V_p^{\partial K}(K)$, where
\begin{align}
&v\in V_p^{K}(K) \mbox{ if and only if }\Delta v\in\PP_{p-2}(K)\mbox{
  in }K\mbox{ and } v=0\mbox{ on } \partial K~,\label{LocalSpaceInterior}\\
&v\in V_p^{\partial K}(K) \mbox{ if and only if }\Delta v=0\mbox{
  in }K\mbox{ and } v\vert_{\partial K}\in \PP_p(\partial K)~.\label{LocalSpaceBoundary}
\end{align}
The dimension of $V_p(K)$ is
\begin{align}\label{DimensionVPK}
\dim V_p(K)=\dim V_p^{K}(K)+\dim V_p^{\partial K}(K)=\dim
  \PP_{p-2}(K)+\dim \PP_p(\partial K)=\binom{p}{2}+\dim \PP_p(\partial K)~.
\end{align}
The dimension of $\PP_p(\partial K)$ depends on the number and nature
of the edges of $K$.  If $e$ is a straight edge, $\dim\PP_p(e)=p+1$,
but if $e$ is not a straight edge, the dimension of $\PP_p(e)$ can be
as high as $\binom{p+2}{2}$, as it is when $p=1$.  The dimension of
$\PP_p(e)$ more generally is given in the following proposition, a
proof of which may be found in~\cite[Theorem 7.1]{McLeodBook}, for example.
\begin{proposition}\label{PpeDimension}
Suppose that $f_m\in\PP_m(\RR^2)$ is an irreducible polynomial of
degree $m$, and that all points $x\in e$ satisfy $f_m(x)=0$.  It holds
that $\dim\PP_p(e)=\binom{p+2}{2}-\binom{p-m+2}{2}$.  If $e$ does not
lie on a real algebraic curve in the plane, then $\dim\PP_p(e)=\binom{p+2}{2}$.
\end{proposition}

\begin{figure}
\begin{center}
  \includegraphics[width=1.8in]{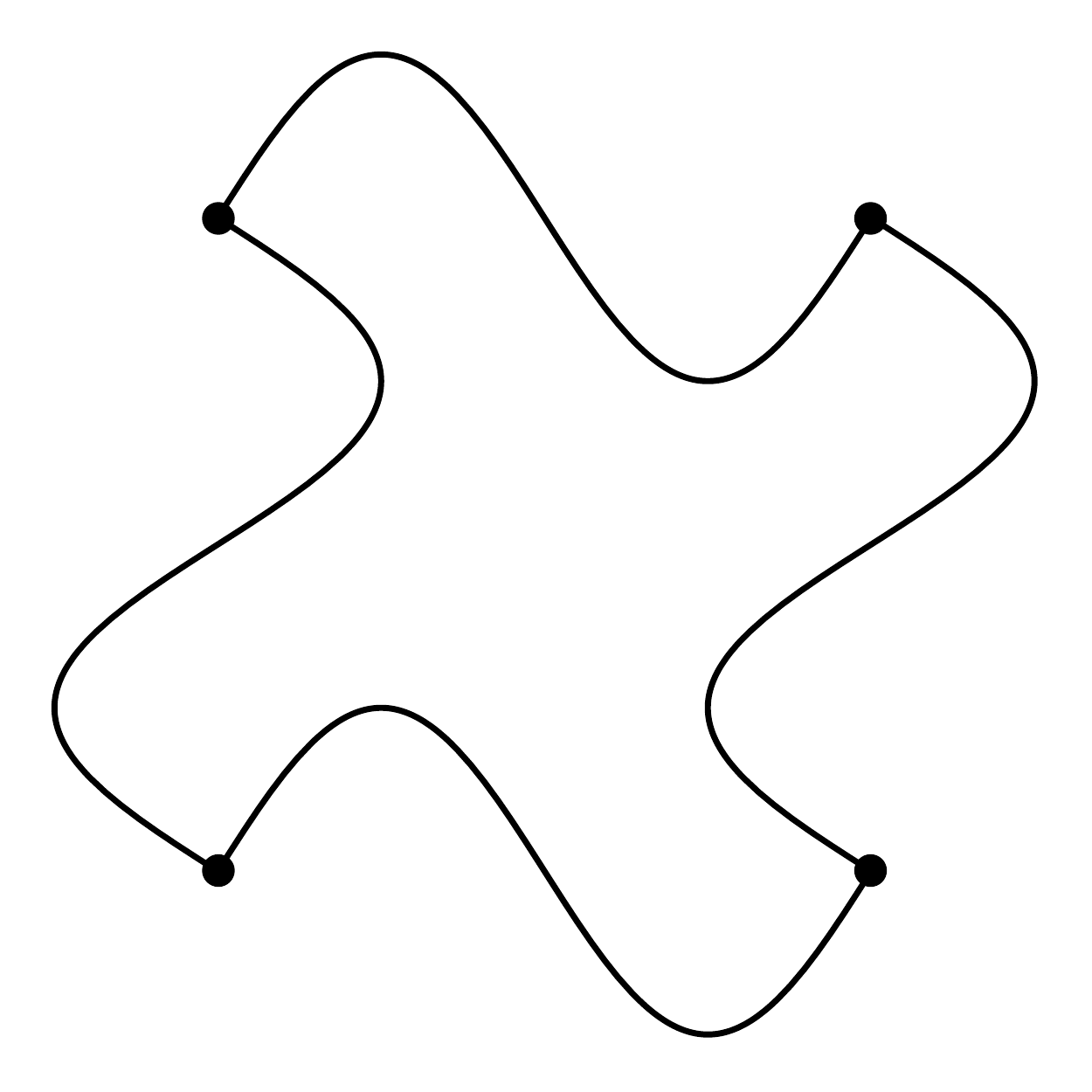}\quad
   \includegraphics[width=2.0in]{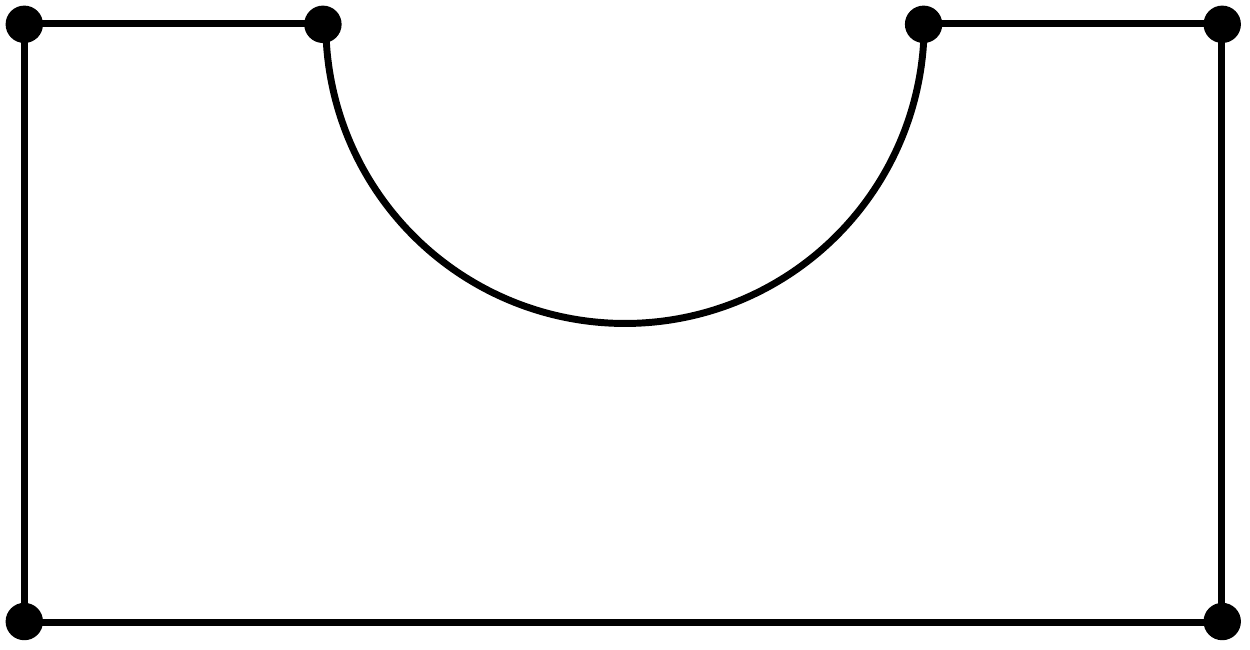}\quad
  \includegraphics[width=1.8in]{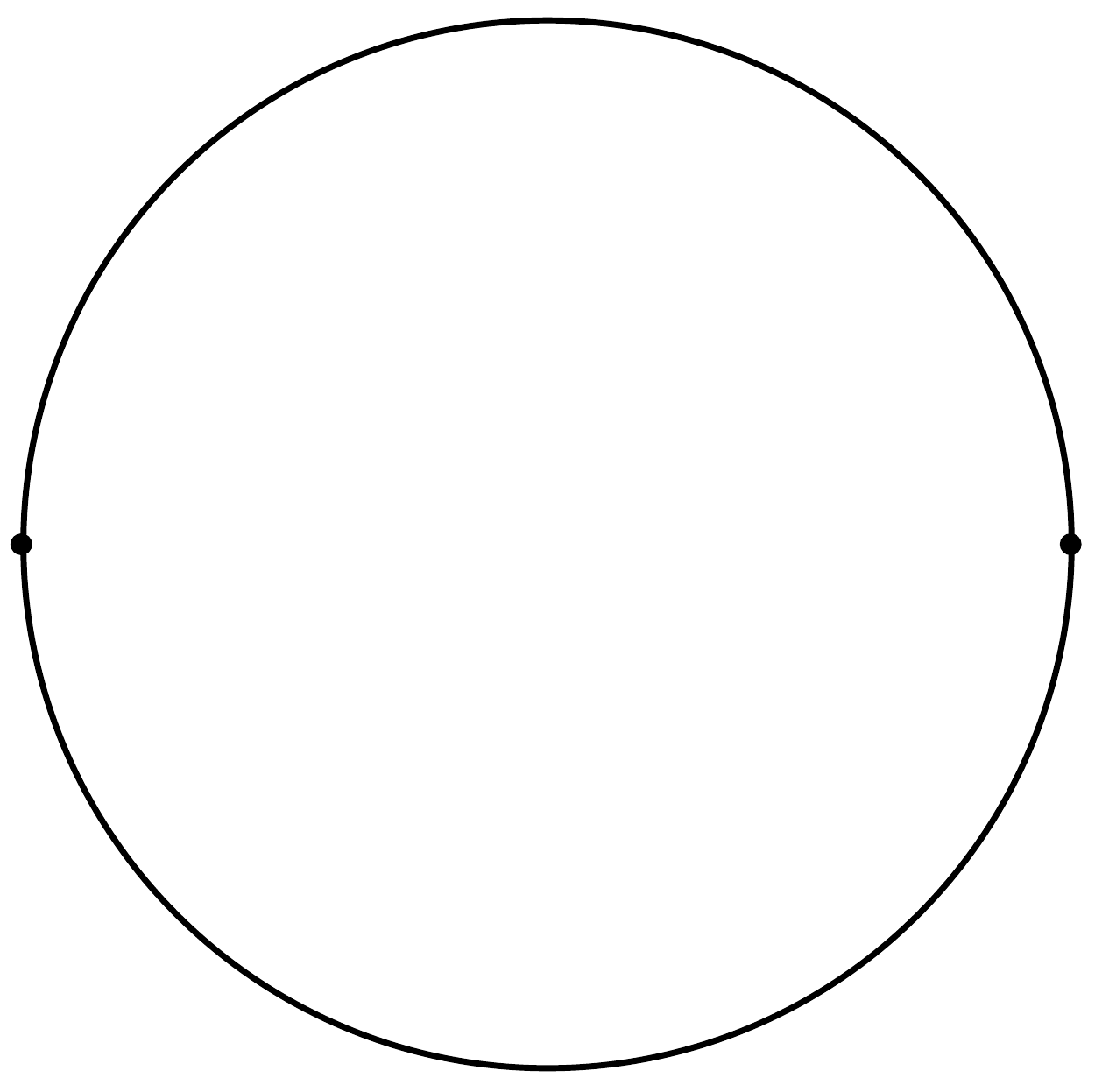}
\end{center}
\caption{\label{ThreePolygons} Shuriken (quadrilateral), Half-Washer
  (hexagon), Two-Edge Circle (bigon). }
\end{figure}
\begin{example}\label{P12Spaces}
  We consider the dimensions of the spaces
  $V_p(K)=V_p^{K}(K)\oplus V_p^{\partial K}(K)$, $p=1,2$, for each of
  the three mesh cells in Figure~\ref{ThreePolygons}.  We have $\dim
  V_1^{K}(K)=0$ and $\dim V_2^{K}(K)=1$.  The continuity of functions in
  $\PP_p(\partial K)$ implies that
\begin{align}\label{DimensionVpBoundary}
\dim V_p^{\partial K}(K)=\dim \PP_p(\partial K)=\sum_{e\subset\partial
  K}\dim\PP_p(e) - \#\mbox{ edges}~.
\end{align}
This formula holds for arbitrary $p$.  For $p=1,2$, we have
\begin{align*}
\dim V_1^{\partial K}(K)&=%2(\#\mbox{ straight edges})+3(\#\mbox{
                         %curved edges})-\#\mbox{ edges}=
(\#\mbox{ straight edges})+2(\#\mbox{ curved edges})~,\\
\dim V_2^{\partial K}(K)&=%3(\#\mbox{ straight edges})+5(\#\mbox{
                         %curved conic edges})+6(\#\mbox{
                         %curved non-conic edges})-\#\mbox{ edges}=
2(\#\mbox{ straight edges})+4(\#\mbox{ curved conic edges})+5(\#\mbox{
                          curved non-conic edges})~.
\end{align*}
For the Half-Washer and Two-Edge Circle, the curved edges are circular
arcs.  For the Shuriken, the curved edges are not
segments of curved conic sections (ellipses, parabolas, hyperbolas).
% Shuriken: a=1/4; x[t_] = {t, a Sin[2 Pi t]}; {t,0,1};
% other edges rigid body  motions of this 
Therefore, the dimensions of these spaces are
\begin{center}
%\begin{small}
\begin{tabular}{|c|ccc|}\hline
                      &Shuriken&Half-Washer&Two-Edge Circle\\\hline
$\dim V_1(K)$&0+(0)+2(4)=8&0+(5)+2(1)=7&0+(0)+2(2)=4\\
$\dim
  V_2(K)$&1+2(0)+4(0)+5(4)=21&1+2(5)+4(1)+5(0)=15&1+2(0)+4(2)+5(0)=9\\\hline
\end{tabular}
%\end{small}
\end{center}
\end{example}

A basis for $\PP_{p-2}(K)$ implicitly defines a basis for
$V_p^{K}(K)$.  In Section~\ref{Computation}, we describe how we form
the associated local finite element linear systems over $V_p(K)$,
using integral equations to get the relevant information about our
basis functions.  At this stage, we merely state that it is convenient
to compute harmonic functions in this context.  To this end, let $z$
be a point in $K$, and $\alpha=(\alpha_1,\alpha_2)$ be a multi-index.
In~\cite{KarachikAntropova2010}, the authors provide an explicit
formula for a polynomial $q_\alpha\in\PP_p(K)$ satisfying
$\Delta q_\alpha=(x-z)^\alpha$; see also the beginning of
Section~\ref{Computation}.  A basis of $V_p^K(K)$,
$\{\phi_\alpha^K\in\PP_p(K):\,|\alpha|\leq p\}$, is given by
$\phi_\alpha^{K}=\psi_\alpha^{K}+q_\alpha$, where
\begin{align}\label{ElementBubble}
%\phi_\alpha^{K}=\psi_\alpha^{K}+q_\alpha\quad,\quad
\Delta \psi_\alpha^{K}=0\mbox{ in }K\quad,\quad
 \psi_\alpha^{K}=-q_\alpha\mbox{ on }\partial K~.
\end{align}

Similarly, a basis of $\PP_p(\partial K)$ naturally leads to a basis
of $V^{\partial K}_p(K)$.  Given an edge $e$ in the mesh, we describe
an approach for obtaining a basis of $\PP_p(e)$ that is independent of
the mesh cell(s) of which it is an edge.  Let $e$ have vertices
$z_0,z_1$.  We choose a third point $z_2$ such that $z_0,z_1,z_2$ are
the vertices of an equilateral triangle (see
Figure~\ref{TraceSpace})---note that $z_2$ typically has nothing to do
with the underlying mesh $\cT$.  Given a global numbering of the
vertices of the mesh, this can be done in a consistent way by choosing
$z_2$ such that a counter-clockwise traversal of the boundary of the
triangle is consistent with traversing the edge $e$ from its smaller
to its larger vertex numbers.  Let
$\ell_0,\ell_1,\ell_2\in\PP_1(\RR^2)$ be the three barycentric
coordinates associated with these vertices.  Formulas for these three
functions are given by
\begin{align}\label{BarycentricFormulas}
\ell_j(x)=1-\frac{(x-z_j)\cdot R(z_{j-1}-z_{j+1})}{(\sqrt{3}/2)h^2}~,
\end{align}
where we understand the subscripts modulo $3$ (i.e. $z_{-1}=z_2$ and
$z_3=z_0$), and 
\begin{align*}
R=\begin{pmatrix}0&1\\-1&0\end{pmatrix}\quad,\quad
                          h=|z_1-z_0|\quad,\quad
                          z_2=\frac{z_1+z_0}{2}-R\cdot \frac{\sqrt{3}(z_1-z_0)}{2}~.
\end{align*}
Any basis for
$\PP_p(\RR^2)$ yields a spanning set for $\PP_p(e)$ by restriction,
and such a basis may be expressed in terms of linear
combinations of products of the barycentric coordinates.  We will consider
hierarchical bases expressed in
this way
(cf.~\cite{Adjerid2001,Beuchler2006,BabushkaSzaboBook}). 
%% Also Carnevalli et al 1993, Schwab book 1998
For example, a hierarchical basis for $\PP_3(\RR^2)$ is
\begin{align}\label{P3Basis}
\left\{\ell_0,\ell_1,\ell_2, 4\ell_1\ell_2, 4\ell_0\ell_2,4\ell_0\ell_1,\frac{3\sqrt{3}}{2}\ell_1\ell_2(\ell_1-\ell_2), \frac{3\sqrt{3}}{2}\ell_0\ell_2(\ell_0-\ell_2), \frac{3\sqrt{3}}{2}\ell_0\ell_1(\ell_0-\ell_1),
  27\ell_0\ell_1\ell_2\right\}~,
\end{align}
Here, we have chosen the scaling on each function so that its maximum
value on the triangle, in magnitude, is $1$.  In any hierarchical
basis for $\PP_p(\RR^2)$, the only functions that do not vanish at
both $z_0$ and $z_1$ are $\ell_0$ and $\ell_1$.  A simple consequence
of this fact is that
\begin{proposition}
For any edge $e$, a hierarchical basis of $\PP_p(e)$ contains both
$\ell_0$ and $\ell_1$.
\end{proposition}

As stated in Proposition~\ref{PpeDimension}, if $e$ lies on an
algebraic curve of order $m$, we know the dimension of $\PP_p(e)$.
However, it may be undesirable to make this determination in practice.
Regardless, we need a practical method for paring down a spanning set
for $\PP_p(e)$ to a basis.  Let $N=\binom{p+2}{2}$, and suppose that
$\{\ell_0,\ell_1,b_1,\ldots,b_{N-2}\}$ is a hierarchical spanning set
of $\PP_p(e)$, as described above.  The functions are listed in
increasing order of degree. The Gram matrix $m_{ij}=\int_e b_ib_j\,ds$
may be used to determine the remaining basis functions (in addition to
$\ell_0,\ell_1$) for $\PP_p(e)$.  We recall that
$\mathrm{rank}(M)=\dim\mathrm{span}\{b_1,\ldots,b_{N-2}\}$
(cf.~\cite[Theorem 7.2.10]{HornJohnsonBook}).  A basis for
$\mathrm{span}\{b_1,\ldots,b_{N-2}\}$  consisting of some subset of
these functions may be determined using a rank-revealing Cholesky
decomposition of $M$ (cf.~\cite{Higham1990,Hansen2001,Gu2004}). 
%, \cite[Sections 4.2.7-4.2.8]{GvL2013}).
We state a
slightly more general version of this result in the following
proposition, and then provide a simple algorithm for selecting a basis
of $\mathrm{span}\{b_1,\ldots,b_{N-2}\}$, and hence of $\PP_p(e)$.
\begin{proposition}\label{BasisSelection}  Let $m_{ij}=\ip{b_j}{b_i}$,
  $1\leq i,j\leq n$ be the Gram matrix associated with an
  inner-product $\ip{\cdot}{\cdot}$ and a list of vectors $(b_1,\dots,
  b_n)$.  Let $P^TMP=R^TR$ be a rank-revealing Cholesky decomposition,
  where $P$ is a permutation matrix, and $R=\begin{pmatrix}R_{11}&R_{12}\end{pmatrix}$,
%$R=\begin{pmatrix}R_{11}&R_{12}\\0&0\end{pmatrix}$,
 with the $r\times r$ matrix $R_{11}$ having strictly positive
 entries.  Then $\{b_{p(j)}:\,1\leq j\leq r\}$ is a basis for
 $\mathrm{span}\{b_1,\ldots,b_{n}\}$, where $p$ is the permutation on
 $\{1,\ldots,n\}$ defined by $P\mb{e}_j=\mb{e}_{p(j)}$, and
 $\{\mb{e}_1,\ldots,\mb{e}_n\}$ are the standard coordinate vectors.
\end{proposition}
\iffalse
is determined by successively
subtracting symmetric rank-one matrices to symmetrically eliminate
rows and columns, always using the largest diagonal entry of the
current matrix (which is the largest entry in magnitude of that
matrix) as the pivot, until the entries of what remains drop below
some small tolerance---in exact arithmetic, .  This is essentially Gaussian
elimination with complete pivoting to obtain a rank-revealing Cholesky
decomposition of (a permuted version of) $M$~\cite[Lemma
1.1]{Higham1990}, see also~\cite[Sections 4.2.7-4.2.8]{GvL2013} .  The indices of the pivots are the indices of the
functions that we keep for our basis.  
\fi

The following algorithm is essentially Gaussian elimination
with complete pivoting for positive semi-definite matrices, where the pivoting is done in place.
%  Townsend, Alex and Trefethen, Lloyd N., Gaussian elimination as an iterative algorithm,
% SIAM News, Vol 46, No 2, March 2013
\begin{algorithm}\label{GECP} Let $m_{ij}=\ip{b_j}{b_i}$,
  $1\leq i,j\leq n$, be the Gram matrix associated with an
  inner-product $\ip{\cdot}{\cdot}$ and a list of vectors $(b_1,\dots,
  b_n)$.  Upon
termination of the following algorithm, $\mathrm{index}=\{p(j):\,1\leq j\leq r\}$, the indices of
a basis $\{b_{p(j)}:\, 1\leq j\leq r\}$ of  $\mathrm{span}\{b_1,\ldots,b_{n}\}$:
\begin{center}
\begin{minipage}{3.0in}
\begin{itemize}
\item[] $\mathrm{index}=\{\}$
\item[] $k=\argmax\{m_{jj}:\,1\leq j\leq n\}$
\item[]  while $m_{kk}>0$
\begin{itemize}
\item[] $\mathrm{index}=\mathrm{index}\cup\{k\}$
\item[] $M= M-m_{kk}^{-1}\mb{m}_k\mb{m}_k^T$
\item[] $k=\argmax\{m_{jj}:\,1\leq j\leq n\}$
\end{itemize} 
\item[] end
\end{itemize}
\end{minipage}
\end{center}
Here, $\mb{m}_k$ is the $k^{th}$ column of the current $M$.  
\end{algorithm}
In practice, one replaces the condition $m_{kk}>0$ with $m_{kk}>\tau$
for some suitably small tolerance $\tau>0$.  Some speed-up of this
basic algorithm may be achieved by exploiting the fact that previous
reduction reduction steps, $M= M-m_{kk}^{-1}\mb{m}_k\mb{m}_k^T$, have
zeroed out the rows and columns in the index set, so these are no
longer needed for further reductions.  In the following example, we
use $\tau=10^{-12}$ in our determination of a basis for $\PP_3(e)$.

\begin{example}\label{CubicEdgeSpaces}
For any edge $e$, a hierarchical spanning set for $\PP_3(e)$ is given by~\eqref{P3Basis}
where we have restricted the domains of these functions to $e$.
Let $e$ be parameterized by $x(t)=(\cosh t,(\sinh t)/2)$, $0\leq t\leq
1$, so $e$ is part of the hyperbola $x^2-4y^2=1$.  We know
in advance that $\dim\PP_3(e)=10-3=7$, and $\ell_0,\ell_1$ will be
part of our basis for $\PP_3(e)$, so we must select five of the
remaining eight functions,
\begin{align*}
\left\{\ell_2, 4\ell_1\ell_2, 4\ell_0\ell_2,4\ell_0\ell_1,\frac{3\sqrt{3}}{2}\ell_1\ell_2(\ell_1-\ell_2), \frac{3\sqrt{3}}{2}\ell_0\ell_2(\ell_0-\ell_2), \frac{3\sqrt{3}}{2}\ell_0\ell_1(\ell_0-\ell_1),
  27\ell_0\ell_1\ell_2\right\}~,
\end{align*}
 to complete our basis.  Taking the functions in this order, and
 forming the associated Gram matrix, we determine that the indices are
 (given in the order computed): 4, 7, 8, 3, 5;  the knowledge that we
 only needed five functions was not used in this computation.
%% See CubicEdgeSpaces.m or Figures.nb
Therefore, our basis
 for $\PP_3(e)$ is given by 
\begin{align*}
\left\{\ell_0,\ell_1,4\ell_0\ell_2,4\ell_0\ell_1,\frac{3\sqrt{3}}{2}\ell_1\ell_2(\ell_1-\ell_2), \frac{3\sqrt{3}}{2}\ell_0\ell_1(\ell_0-\ell_1),
  27\ell_0\ell_1\ell_2\right\}~.
\end{align*}
These basis functions are plotted, as functions of the parameter $t$,
in Figure~\ref{TraceSpace}, together with the edge $e$ and associated triangle
used to define the barycentric coordinates $\ell_0,\ell_1,\ell_2$.
\begin{figure}
\begin{center}
\includegraphics[width=2.0in]{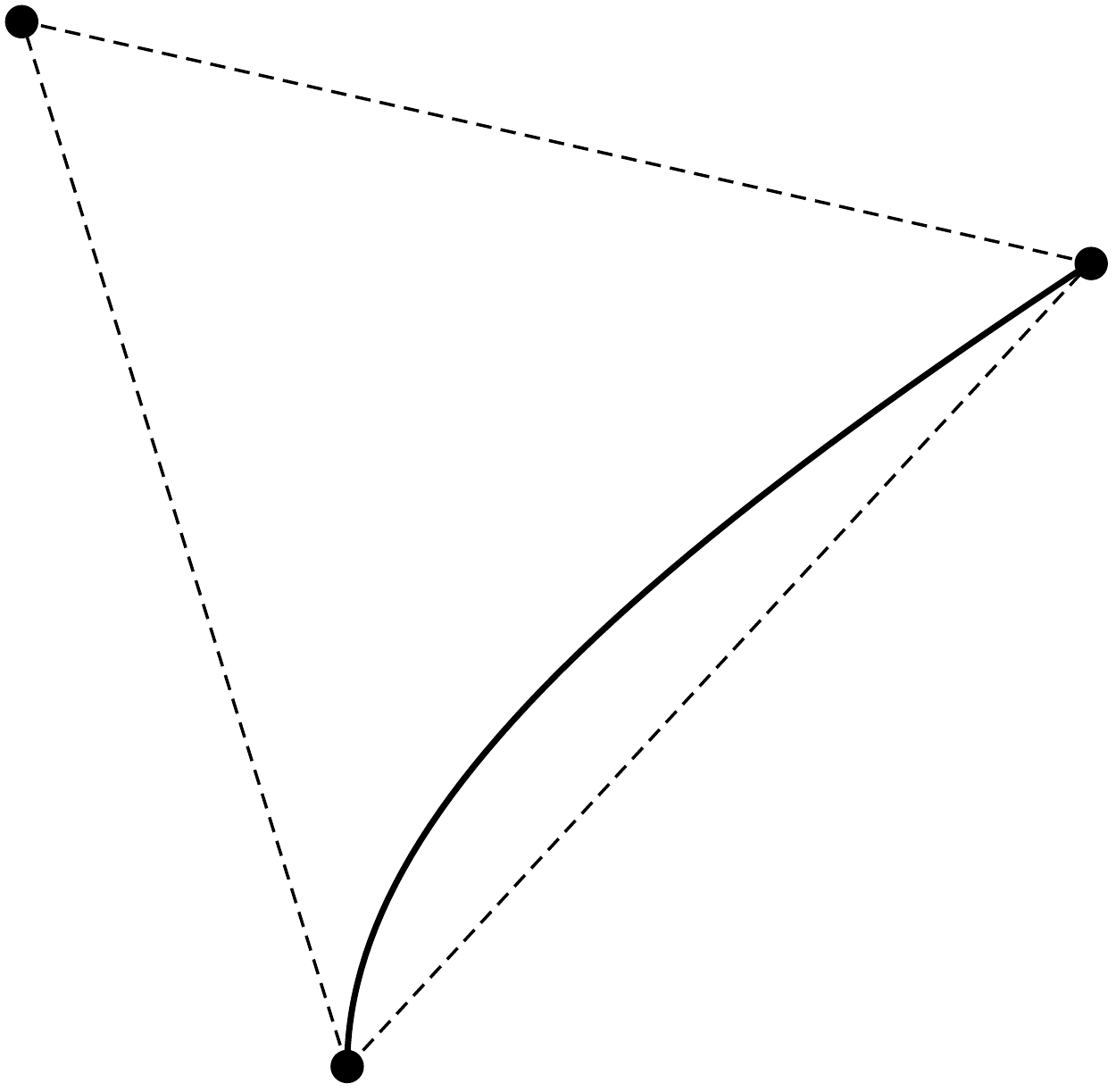}
\includegraphics[width=3.0in]{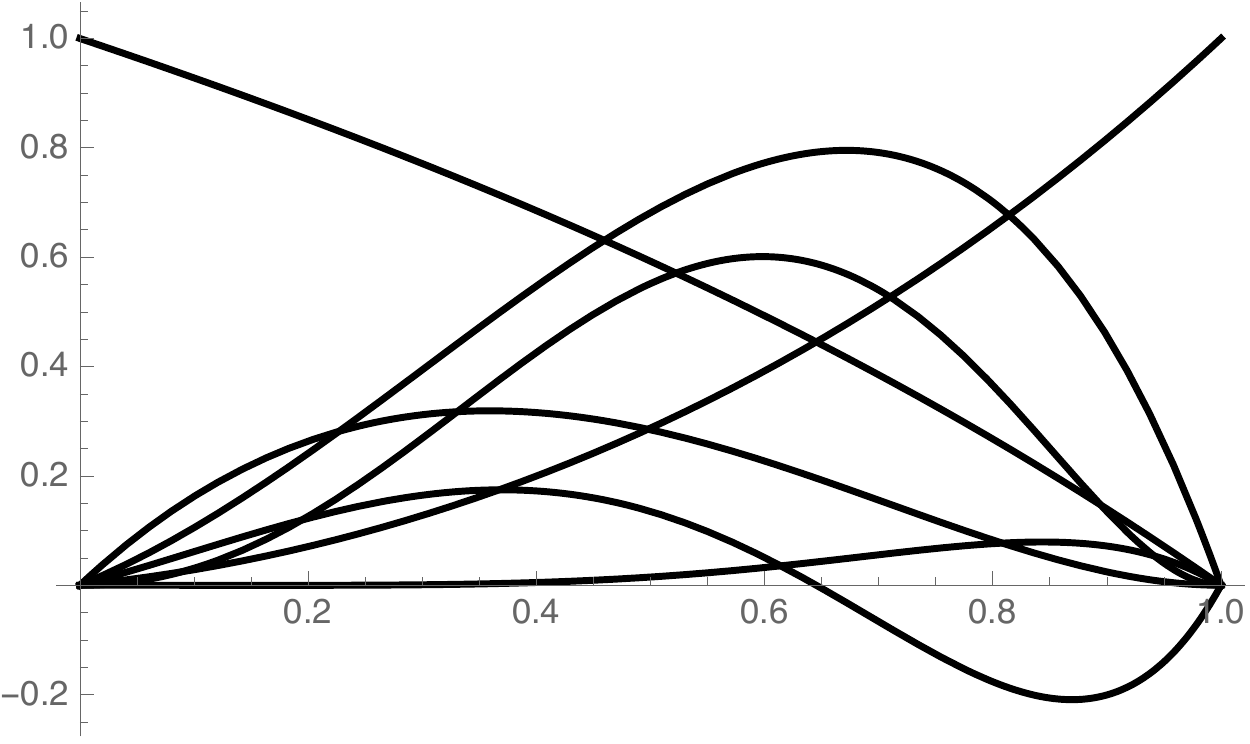}
\end{center}
\caption{\label{TraceSpace} At left,
  the edge $e$ and associated triangle (dashed) for Example~\ref{CubicEdgeSpaces}.  At right, plots of a
basis for $\PP_3(e)$ with respect to a parametrization, $x=x(t)$, of $e$.}
\end{figure}
As a matter of interest, we note that, when the reduction algorithm
was used, with $\tau=10^{-12}$ as before, on the entire spanning set
\eqref{P3Basis}, a different basis was obtained,
\begin{align*}
\left\{\ell_0,\ell_1, 4\ell_1\ell_2, 4\ell_0\ell_2,4\ell_0\ell_1,\frac{3\sqrt{3}}{2}\ell_1\ell_2(\ell_1-\ell_2), \frac{3\sqrt{3}}{2}\ell_0\ell_2(\ell_0-\ell_2)\right\}~.
\end{align*}
\end{example}

\begin{remark}\label{DiagonalRescaling}
A natural variant of Algorithm~\ref{GECP} that may be used if the
diagonal entries of $M$ are all non-zero is to diagonally rescale its
entries, $m_{ij}\longleftarrow m_{ij}/\sqrt{m_{ii}m_{jj}}$, before
beginning the elimination loop.  If we do this for
Example~\ref{CubicEdgeSpaces}, the resulting basis for $\PP_3(e)$ is
\begin{align*}
\left\{\ell_0,\ell_1,\ell_2, \frac{3\sqrt{3}}{2}\ell_1\ell_2(\ell_1-\ell_2), \frac{3\sqrt{3}}{2}\ell_0\ell_2(\ell_0-\ell_2), \frac{3\sqrt{3}}{2}\ell_0\ell_1(\ell_0-\ell_1),
  27\ell_0\ell_1\ell_2\right\}~,
\end{align*}
regardless of whether we use the entire spanning set~\eqref{P3Basis},
or remove $\{\ell_0,\ell_1\}$, in constructing $M$.
\end{remark}

\begin{remark}\label{AltBasis}
We describe an alternative to the barycentric
coordinates~\eqref{BarycentricFormulas} associated with $e$ that acts
more like a local cartesian coordinate system.  Using the same
notation $z_0,z_1$, $h$ and $R$, we define
\begin{align}\label{AltBary}
\tilde{\ell}_0(x)=\frac{(z_1-x)\cdot(z_1-z_0)}{h^2}\quad,\quad
\tilde{\ell}_1(x)=\frac{(x-z_0)\cdot(z_1-z_0)}{h^2}\quad,\quad
\tilde{\ell}_2(x)=\frac{(z_1-x)\cdot R(z_1-z_0)}{h^2}~.
\end{align} 
Straight-forward manipulations reveal that
\begin{align*}
\ell_0=\tilde\ell_0-\frac{\tilde\ell_2}{\sqrt{3}}\quad,\quad
\ell_1=\tilde\ell_1-\frac{\tilde\ell_2}{\sqrt{3}}\quad,\quad
\ell_2=\frac{\tilde{\ell}_2}{\sqrt{3}/2}~,
\end{align*}
so it is simple to translate between coordinate systems if desired.
\end{remark}

Having now properly defined $V_p(\cT)\subset H^1(\Omega)$, the discrete version
of~\eqref{ModelProblem} is to find $\hat{u}\in V_p(\cT)\cap\cH$ such that 
\begin{align}\label{DiscModelProblem}
\int_\Omega A\nabla \hat{u}\cdot\nabla v+(\mb{b}\cdot\nabla{u}+c\hat{u})v\,dx=\int_\Omega f v\,dx+\int_{\partial\Omega_N}gv\,ds\mbox{
for all }v\in V_p(\cT)\cap\cH~.
\end{align}
The intersection, $V_p(\cT)\cap\cH$, ensures that we respect any
homogeneous Dirichlet boundary conditions inherent in $\cH$.  
For standard finite elements, as well as those defined on more general
polygonal meshes,
common assumptions on the data ensure that the finite element error
$\|u-\hat{u}\|_{H^1(\Omega)}$ is controlled by interpolation error
$\|u-\cI u\|_{H^1(\Omega)}$, where $\cI u\in V_p(\cT)$ is some
appropriately defined interpolant of $u$.  In the next section, we
define a projection-based interpolation operator appropriate for our
setting, and prove that it yields the desired approximation properties.

\section{Interpolation}\label{Interpolation}
In this section, we describe a local interpolation scheme 
\begin{align}
\cI_K: W(K)=\{v\in C(\overline K)\cap H^1(K):\,\Delta v\in L^2(K)\}\to V_p(K)~,
\end{align}
and establish local error estimates under stronger regularity
assumptions.  By construction, the local interpolation operator will
define a global interpolation operator
\begin{align}
\cI: W=\{v\in C(\overline \Omega)\cap H^1(\Omega):\,\Delta v\in L^2(\Omega)\}\to V_p(\cT)
\end{align}
by $(\cI v){\vert_{K}}=\cI_K v$.

Our definition of $\cI_K$ is motivated by the decomposition $V_p(K)=V_p^K(K)\oplus V_p^{\partial K}(K)$.
We begin with a related decomposition of $v$ as $v=v^K+v^{\partial K}$, where 
\begin{align}\label{vDecomposition}
\begin{cases}\Delta v^K=\Delta v&\mbox{ in }K\\v^K=0&\mbox{ on }\partial K\end{cases}
\quad,\quad
\begin{cases}\Delta v^{\partial K}=0&\mbox{ in }K\\v^{\partial K}=v&\mbox{ on }\partial K\end{cases}
~.
\end{align}
We define $\cI_K$ by an analogous decomposition
$\cI_K v=\cI_K^K v+\cI_K^{\partial K}v$, where $\cI_K^K v\in V_p^K(K)$ and $\cI_K^{\partial K} v\in V_p^{\partial K}(K)$ are given by
\begin{align}\label{LocalInterpolant1}
\begin{cases}\Delta (\cI_K^K v)=q^K&\mbox{ in }K\\\cI_K^K v=0&\mbox{ on }\partial K\end{cases}
\quad,\quad
\begin{cases}\Delta (\cI_K^{\partial K} v )=0&\mbox{ in }K\\\cI_K^{\partial K} v=q^{\partial K}&\mbox{ on }\partial K\end{cases}
~.
\end{align}
In order to complete this definition, we must define
$q^K\in\PP_{p-2}(K)$ and $q^{\partial K}\in\PP_p(\partial K)$.  We
define $q^K$ by
\begin{align}\label{LocalInterpolant2}
\int_K (\Delta v - q^K) \phi\,dx=0\mbox{ for all }\phi\in \PP_{p-2}(K)~.
\end{align}
We define  $q^{\partial K}$ by defining it on each edge of
$\partial K$.
For a non-trivial open subset $\Gamma\subset \partial K$, we use the inner-product
\begin{align}\label{H1_2InnerProduct}
(\phi,\psi)_{H^{1/2}(\Gamma)}=\int_{\Gamma}\phi\psi\,ds+
\int_{\Gamma}\int_{\Gamma}\frac{(\phi(x)-\phi(y)) (\psi(x)-\psi(y))}{|x-y|^2}\,ds(x)\,ds(y)~,
\end{align}
with $\|\cdot\|_{H^{1/2}(\Gamma)}$ as the associated norm.
%%  See Definition 1.3.3.2 and (1,3,3,3) of Grisvard \cite{MR3396210}
Below, we take $\Gamma$ to be either the entire boundary,
$\partial K$, or a single edge, $e$.  Fix an edge $e$ of $\partial K$, having endpoints $z_0,z_1$, and let
$\PP_{p,0}(e)=\{w\in\PP_p(e):\,w(z_0)=w(z_1)=0\}$.  We define $q_e\in\PP_m(e)$
by the conditions
\begin{align}\label{LocalInterpolant3}
q_e(z_0)=v(z_0)\quad,\quad q_e(z_1)=v(z_1)\quad,\quad (v-q_e,w)_{H^{1/2}(e)}=0
\mbox{ for all }w\in \PP_{p,0}(e)~.
\end{align}
Finally, $q^{\partial K}$ is defined by $(q^{\partial K}){\vert_e}=q_e$.

In several places below, it will be convenient to use the following
basic result.  Suppose that $\psi,\phi\in H^1(K)$, and $\Delta \psi=0$ in $K$ and
$\phi=0$ on $\partial K$.  Then
$\int_K\nabla\psi\cdot\nabla\phi\,dx=0$, so 
$|\psi+\phi|_{H^1(K)}^2=|\psi|_{H^1(K)}^2 +|\phi|_{H^1(K)}^2$.  For example, we have
\begin{align}\label{InterpolationPythagorean}
|v-\cI_K v|_{H^1(K)}^2= |v^K-\cI_K^K v|_{H^1(K)}^2+|v^{\partial K}-\cI_K^{\partial K} v|_{H^1(K)}^2~,
\end{align}
and we consider both contributions to the interpolation error in turn.
In the proofs below, we use $c$ as a constant that may vary from one
appearance to the next.   Throughout, $h_K$ denotes the diameter of $K$.
 We first consider $v^K-\cI_K^K v$.

\begin{proposition}\label{InterpolationBound1}
Suppose that $v\in H^1(K)$ and $\Delta v\in H^{p-1}(K)$ for some
$p\geq 1$.  There is a
scale-invariant constant $c=c(p,K)>0$ for which
\begin{align*}
|v^K-\cI_K^K v|_{H^1(K)}&\leq c h_K^{p} |\Delta v|_{H^{p-1}(K)}
\quad,\quad 
\|v^K-\cI_K^K v\|_{L^2(K)}\leq c h_K^{p+1} |\Delta v|_{H^{p-1}(K)}~.
\end{align*}
\end{proposition}
\begin{proof}
It holds that
\begin{align*}
|v^K-\cI_K^K v|_{H^1(K)}^2&
%=\int_{\partial K}\frac{\partial (v^K-\cI_K^K v)}{\partial \mb{n}}\, (v^K-\cI_K^K v)\,ds
%-\int_K \Delta (v^K-\cI_K^K v) \, (v^K-\cI_K^K v)\,dx\\
=-\int_K (\Delta v-q^K )  (v^K-\cI_K^K v)\,dx
\leq \|\Delta v-q^K \|_{L^2(K)} \| v^K-\cI_K^K v\|_{L^2(K)}~.
\end{align*}
Since $v^K-\cI_K^K v\in H^1_0(K)$, the Poincar\'e-Friedrichs
Inequality, $\|w\|_{L^2(K)}\leq h_K|w|_{K}$ for $w\in H^1_0(K)$, 
\iffalse
For $w\in H^1_0(K)$ and any $z\in\RR^2$, $\int_K w^2\,dx=-\int_Kw
\nabla w\cdot (x-z)\,dx\leq \|w\|_{L^2(K)}\|w\cdot
(x-z)\|_{L^2(K)}\leq \|w\|_{L^2(K)}|w|_{H^1(K)} h_K$.
\fi
ensures that
\begin{align}\label{Friedrichs}
|v^K-\cI_K^K v|_{H^1(K)}&\leq h_K\|\Delta v-q^K \|_{L^2(K)} =h_K \inf_{\phi\in \PP_{p-2}(K)}\|\Delta v-\phi\|_{L^2(K)} ~.
\end{align}
The estimate $|v^K-\cI_K^K v|_{H^1(K)}\leq c h_K^{p} |\Delta v|_{H^{p-1}(K)}$
follows from this by applying the Bramble-Hilbert Lemma. 
%% Lemma 4.3.8 in Brenner-Scott book
The $L^2(K)$ norm result follows from this by applying the
Poincar\'e-Friedrichs Inequality again.
\end{proof}
\begin{remark}
  If $K$ is convex, $h_K$ can be replaced by $h_K/\pi$
  in~\eqref{Friedrichs} (cf. \cite{Payne1960}).  Furthermore, for convex
  $K$, the dependence on $K$ of the constant $c(p,K)$ coming from the
  Bramble-Hilbert Lemma in
  Proposition~\ref{InterpolationBound1} can be removed, and for
  non-convex domains that are star-shaped with respect to a point,
  ball, or more general subdomain, various estimates of how $c(p,K)$
  depends on the shape of $K$ have been
  established~\cite{MR1726481,MR714693,MR2050198,MR2290609}.
% and further relaxations of these constraints have also been achieved~\cite{MR2922624} .
\end{remark}

For our analysis of $v^{\partial K}-\cI_K^{\partial K}v$ the following
result will be useful.
\begin{proposition}\label{LInfinityProp}  If $v\in H^2(K)$ and $\Delta v = 0$ in $K$, 
there is a scale-invariant constant $c=c(K)$ for which
\begin{align}\label{LInfinityBound}
\|v\|_{L^\infty(K)}\leq c\inf\{h_K|w|_{H^2(K)}+|w|_{H^1(K)}+h_K^{-1}\|w\|_{L^2(K)}:\; w\in
  H^2(K)\mbox{ and }w=v\mbox{ on }\partial K\}~.
\end{align}
\end{proposition}
\begin{proof}
For $w\in H^2(K)$, we have $\|w\|_{L^{\infty}(K)}\leq c\|w\|_{H^2(K)}$
by a Sobolev embedding result.  A standard scaling argument then yields
\begin{align*}
\|w\|_{L^\infty(K)}\leq c\left(h_K|w|_{H^2(K)}+|w|_{H^1(K)}+h_K^{-1}\|w\|_{L^2(K)} \right)~,
\end{align*}
where $c=c(K)$ is scale-invariant.  Now~\eqref{LInfinityBound} follows
from the fact that harmonic functions on $K$ attain their extrema on
$\partial K$, so, if $v,w\in H^2(K)$ have the same Dirichlet trace on
$\partial K$, and $v$ is harmonic on $K$, then
$\|v\|_{L^\infty(K)}\leq \|w\|_{L^\infty(K)}$.
\end{proof}

\begin{remark}\label{rem:LInfinityProp}
  Since we are working in $\RR^2$, $H^{1+s}(K)$ is continuously
  imbedded in $C(\overline{K})$ for any $s\in (0,1)$.  Therefore,
  Proposition~\ref{LInfinityProp} is readily generalized to such
  spaces, with the obvious bound
\begin{align}\label{LInfinityBoundB}
  \|v\|_{L^\infty(K)}\leq c\inf\{h_K^s|w|_{H^{1+s}(K)}+|w|_{H^1(K)}+h_K^{-1}\|w\|_{L^2(K)}:\; w\in
  H^{1+s}(K)\mbox{ and }w=v\mbox{ on }\partial K\}~.
\end{align} 
Typical assumptions on the domain $\Omega$ and the data for the
problem guarantee that $u\in H^{1+s}(\Omega)$ for some $s>0$ (cf.
~\cite{MR3396210,Grisvard1992,Wigley1964}).  
\end{remark}

We now consider the term $|v^{\partial K}-\cI_K^{\partial K}v|_{1,K}$.
Let $e$ be an edge of $\partial K$, with endpoints $z_0,z_1$.  We
begin with a further decomposition of $q_e$, namely
$q_e=q_{e,1}+q_{e,0}$, where $q_{e,1}=v(z_0)\ell_0+v(z_1)\ell_1\in\PP_1(e)$ and
$q_{e,0}=q_e-q_{e,1}\in\PP_{p,0}(e)$.  This induces a natural
decomposition of $q^{\partial K}$, $q^{\partial K}=q^{\partial
  K}_1+q^{\partial K}_0$, where $q^{\partial K}_1\in \PP_1(\partial K)$ satisfies 
$q^{\partial  K}_1(z)=v(z)$ at each vertex $z$ of $K$, and
$q^{\partial  K}_0$ vanishes at the vertices.
\begin{proposition}\label{InterpolationBound2}
Suppose that $v\in H^{p+1}(K)$ for some $p\geq 1$. 
There is a scale-invariant constant $c=c(p,K)$ for which
\begin{align*}
|v^{\partial K}-\cI_K^{\partial K} v|_{H^1(K)}&\leq c h_K^{p} |v|_{H^{p+1}(K)}
\quad,\quad 
\|v^{\partial K}-\cI_K^{\partial K} v\|_{L^2(K)}\leq c h_K^{p+1} |v|_{H^{p+1}(K)}~.
\end{align*}
\end{proposition}
\begin{proof}
We decompose $\cI^{\partial K}_Kv$ as $\cI^{\partial K}_Kv=w_1+w_0$,
where
\begin{align*}
\begin{cases}\Delta w_1=0&\mbox{ in }K\\w_1=q_1^{\partial K}&\mbox{ on }\partial K\end{cases}
\quad,\quad
\begin{cases}\Delta w_0=0&\mbox{ in }K\\w_0=q_0^{\partial K}&\mbox{ on }\partial K\end{cases}
~.
\end{align*}
It follows that $|\cI^{\partial K}_Kv|_{H^1(K)}\leq |w_1|_{H^1(K)}+|w_0|_{H^1(K)}$.

We denote the set of vertices of $K$ by $\cV(K)$, and the set of edges
of $K$ by $\cE(K)$.  For $z\in \cV(K)$, we define
$\ell_z\in\PP_1(\partial K)$ as follows: if $e$ is not adjacent to
$z$, then $\ell_z$ vanishes on $e$, and if $e$ is adjacent to $z$,
then $\ell_z=\ell_j$ on $e$, where $z=z_j$ for one of the endpoints
$z_0,z_1$ of $e$. Let $\phi_z$ be the harmonic function on $K$ whose
Dirichlet trace on $\partial K$ is $\ell_z$.  It follows that $w_1=\sum_{z\in\cV(K)}v(z)\phi_z$, so
\begin{align*}
|w_1|_{H^1(K)}\leq \|v^{\partial K}\|_{L^\infty(K)}\sum_{z\in\cV(K)}|\phi_z|_{H^1(K)}
\leq c \|v^{\partial K}\|_{L^\infty(K)}\leq c(h_K|v|_{H^2(K)}+|v|_{H^1(K)}+h_K^{-1}\|v\|_{L^2(K)})~,
\end{align*}
where we have used~\eqref{LInfinityBound} in the final inequality.
A similar argument shows that $\|w_1\|_{L^2(K)}\leq c(h_K^2|v|_{H^2(K)}+ h_K|v|_{H^1(K)}+\|v\|_{L^2(K)})$.

From~\eqref{LocalInterpolant3}, we see that
$(q_{e,0},q_{e,0})_{H^{1/2}(e)}=(v-q_{e,1},q_{e,0})_{H^{1/2}(e)}$, so
$\|q_{e,0}\|_{H^{1/2}(e)}\leq \|v-q_{e,1}\|_{H^{1/2}(e)}$, for each
edge $e$.  Now, 
\begin{align*}
|w_0|^2_{H^1(K)}=\int_{\partial K}(\partial
w_0/\partial n)q_0^{\partial K}\,ds\leq c\|\partial
w_0/\partial n\|_{H^{-1/2}(\partial K)}\|q_0^{\partial  K}\|_{H^{1/2}(\partial K)}
\leq c |w_0|_{H^1(K)}\|q_0^{\partial K}\|_{H^{1/2}(\partial K)}~.
\end{align*}
Here we have used applied the trace inequality $\|\partial
w_0/\partial n\|_{H^{-1/2}(\partial K)}\leq c(|w_0|_{H^1(K)}+\|\Delta w_0\|_{L^2(K)})=c|w_0|_{H^1(K)}$ (cf.~\cite[Theorem A.33]{Schwab1998Book}), where
$c=c(K)$ is scale-invariant.
From this it follows that
\begin{align*}
|w_0|^2_{H^1(K)}\leq c\|q_0^{\partial K}\|^2_{H^{1/2}(\partial K)}\leq
  c \sum_{e\in \cE(K)} \|q_0^{\partial K}\|^2_{H^{1/2}(e)}\leq c
  \sum_{e\in \cE(K)} \|v-q_{e,1}\|^2_{H^{1/2}(e)} 
\leq c \|v-q_{1}^{\partial K}\|^2_{H^{1/2}(\partial K)}~. 
\end{align*}
The second inequality holds because 
$q_0$ vanishes at the vertices, see Remark~\ref{H1_2NormEquivalence}.
At this stage, $c=c(p,K)$.

Another standard trace inequality ensures that
$\|v-q_{1}^{\partial K}\|_{H^{1/2}(\partial K)}\leq c\left(|v-w_1|_{H^1(K)}+h_K^{-1}\|v-w_1\|_{L^2(K)}\right)$.
Combining this with our estimates above, we obtain
\begin{align*}
|w_0|_{H^1(K)}&\leq c\left(|v-w_1|_{H^1(K)}+h_K^{-1}\|v-w_1\|_{L^2(K)}\right)\\
&\leq c\left(h_K|v|_{H^2(K)}+|v|_{H^1(K)}+h_K^{-1}\|v\|_{L^2(K)}\right)~,
\end{align*}
and it follows, by applying the estimates for $|w_0|_{H^1(K)}$ and $|w_1|_{H^1(K)}$, that 
\begin{align}\label{StabilityBound}
|\cI^{\partial K}_Kv|_{H^1(K)}\leq  c\left(h_K|v|_{H^2(K)}+|v|_{H^1(K)}+h_K^{-1}\|v\|_{L^2(K)}\right)~.
\end{align}
A standard inverse inequality, and the fact that our interpolation
scheme preserves constants, yields the obvious analogue in $L^2(K)$,
\begin{align}\label{StabilityBoundL2}
\|\cI^{\partial K}_Kv\|_{L^2(K)}\leq  c\left(h_K^2|v|_{H^2(K)}+ h_K|v|_{H^1(K)}+\|v\|_{L^2(K)}\right)~.
\end{align}

Now, let $\phi\in V_p(K)$ and decompose it as
$\phi=\phi^{K}+\phi^{\partial K}$, with $\phi^{K}\in V_p^{K}(K)$ and
$\phi^{\partial K}\in V_p^{\partial K}(K)$.  We have
$|v-\phi|_{H^1(K)}^2=|v^{K}-\phi^{K}|_{H^1(K)}^2+|v^{\partial
  K}-\phi^{\partial K}|_{H^1(K)}^2$, so $|v^{\partial
  K}-\phi^{\partial K}|_{H^1(K)}\leq |v-\phi|_{H^1(K)}$.  Noting that $\cI^{\partial K}_K
\phi=\phi^{\partial K}$, and applying~\eqref{StabilityBound} to $v-\phi$,
we see that
\begin{align*}
|v^{\partial K}-\cI^{\partial K}_K v|_{H^1(K)}&\leq
|v^{\partial K}-\phi^{\partial K}|_{H^1(K)}+|\cI^{\partial K}_K
                                                (v-\phi)|_{H^1(K)}\\
&
\leq c\left(h_K|v-\phi|_{H^2(K)}+|v-\phi|_{H^1(K)}+h_K^{-1}\|v-\phi\|_{L^2(K)}\right)~.
\end{align*}
Since $\PP_p(K)\subset V_p(K)$,
the Bramble-Hilbert Lemma now implies that $|v^{\partial
  K}-\cI^{\partial K}_K v|_{H^1(K)}\leq c h_K^p|v|_{H^{p+1}(K)}$, as
claimed.  
%Similar reasoning yields the result for the $L^2(K)$ norm as well.
The result for the $L^2(K)$ norm follows the same pattern, but we
briefly lay out the argument anyway.  It holds that
\begin{align*}
\|v^{\partial K}-\cI^{\partial K}_K v\|_{L^2(K)}&\leq
\|v^{\partial K}-\phi^{\partial K}\|_{L^2(K)}+\|\cI^{\partial K}_K
                                                (v-\phi)\|_{L^2(K)}\\
&
\leq \|v^{\partial K}-\phi^{\partial K}\|_{L^2(K)}+ c\left(h_K^2|v-\phi|_{H^2(K)}+h_K|v-\phi|_{H^1(K)}+\|v-\phi\|_{L^2(K)}\right)~.
\end{align*}
It remains to estimate $\|v^{\partial K}-\phi^{\partial
  K}\|_{L^2(K)}$, for which we have
\begin{align*}
\|v^{\partial K}-\phi^{\partial K}\|_{L^2(K)}&\leq
                                               \|v-\phi\|_{L^2(K)}+\|v^{K}-\phi^{K}\|_{L^2(K)}\\
&\leq \|v-\phi\|_{L^2(K)}+h_K|v^{K}-\phi^{K}|_{H^1(K)}\leq \|v-\phi\|_{L^2(K)}+h_K|v-\phi|_{H^1(K)}~.
\end{align*}
Combining this with our previous estimate yields,
\begin{align*}%\label{L2InterpApparox}
\|v^{\partial K}-\cI^{\partial K}_K v\|_{L^2(K)}\leq c\left(h_K^2|v-\phi|_{H^2(K)}+h_K|v-\phi|_{H^1(K)}+\|v-\phi\|_{L^2(K)}\right)~,
\end{align*}
and Bramble-Hilbert Lemma completes the argument.
\end{proof}

\begin{remark}\label{H1_2NormEquivalence}
The claim that $\|q_0^{\partial K}\|_{H^{1/2}(\partial K)}^2\leq c
\sum_{e\in\cE(K)}\|q_0^{\partial K}\|_{H^{1/2}(e)}^2$ in the proof of
Proposition~\ref{InterpolationBound2} requires further comment.
Superficially, this holds because both quantities are (squares of) norms on the
finite dimensional vector space $\PP_{p,0}(\partial K)=\{w\in
P_p(\partial K):\, w(z)=0\mbox{ for all }z\in \cV(K)\}$.  Although
such an argument allows for the dependence of $c$ on $\dim
\PP_{p,0}(\partial K)$ (hence on $p$), we want to ensure that $c$ is
scale-invariant.  For that, we look a little closer at the norms.
We have
$\|w\|_{H^{1/2}(\Gamma)}^2=\|w\|_{L^{2}(\Gamma)}^2+|w|_{H^{1/2}(\Gamma)}^2$,
where 
\begin{align*}
|w|_{H^{1/2}(\Gamma)}^2=\int_{\Gamma}\int_{\Gamma}\frac{(w(x)-w(y))^2}{|x-y|^2}\,ds(x)\,ds(y)~.
\end{align*}
As suggested by the notation, $|\cdot|_{H^{1/2}(\Gamma)}$ is generally
a semi-norm, with constant functions as its kernel.  However, for
$w\in \PP_{p,0}(\partial K)$, both $|w|_{H^{1/2}(\partial K)}$ and
$\left(\sum_{e\in\cE(K)}|w|_{H^{1/2}(e)}^2\right)^{1/2}$ are norms, so
there is a constant $c$ such that $|w|_{H^{1/2}(\partial K)}\leq c
\left(\sum_{e\in\cE(K)}|w|_{H^{1/2}(e)}^2\right)^{1/2}$.
Since both norms in this inequality are scale invariant, so is $c$.
Since $\|w\|_{L^2(\partial K)}^2=\sum_{e\in\cE(K)}\|w\|_{L^{2}(e)}^2$,
we have the result that was claimed.
%% The discussion in Section 3 of 
%% The Construction of Preconditioners for Elliptic Problems by Substructuring. I
%% Bramble, Pasciak, Schatz
%% Math. Comp, Vol. 47, No. 175, pp. 103-134, 1986
%% is also relevant, though not quite what we need here.

We briefly mention two earlier contributions that have considered some
of the same issues that we do here concerning working with the
$H^{1/2}$ norm on all versus individual parts of the boundary of a
mesh cell or a polyhedral subdomain , but in the context of standard
finite element meshes.  The first is~\cite{MR842125}, and it concerns
domain decomposition-type preconditioners for linear solvers. Though
we were unable to use the results of Section 3 in that paper related
to localization of the $H^{1/2}$ norm our context, they provide the
first discussion and treatment of this issue of which we are aware in
the finite element literature. The second contribution
is~\cite{Demkowicz2005}, in which the authors set forth
projection-based interpolation schemes that are conforming in $H^1$,
$H(\mathrm{curl})$ and $H(\mathrm{div})$ spaces.  Our interpolation
scheme is also projection based, but because their results were for
standard element shapes, they could not be readily applied in our
context.
\end{remark}

\begin{remark}\label{InterpBestApprox}
The proof of Proposition~\ref{InterpolationBound2} revealed that
\begin{align*}
|v^{\partial K}-\cI^{\partial K}_K v|_{H^1(K)}&
\leq c\inf_{\phi\in
  V_p(K)}\left(h_K|v-\phi|_{H^2(K)}+|v-\phi|_{H^1(K)}+h_K^{-1}\|v-\phi\|_{L^2(K)}\right)~,\\
\|v^{\partial K}-\cI^{\partial K}_K v\|_{L^2(K)}&
\leq c\inf_{\phi\in
  V_p(K)}\left(h_K^2|v-\phi|_{H^2(K)}+h_K|v-\phi|_{H^1(K)}+\|v-\phi\|_{L^2(K)}\right)~.
\end{align*}
In fact, by the same reasoning as discussed in
Remark~\ref{rem:LInfinityProp}, we have the expected versions for
fractional order spaces as well, for $s\in(0,1]$,
\begin{align}\label{InterpBestApprox_H1}
|v^{\partial K}-\cI^{\partial K}_K v|_{H^1(K)}
&
\leq
  c\inf_{\phi\in
  V_p(K)}\left(h_K^s|v-\phi|_{H^{1+s}(K)}+|v-\phi|_{H^1(K)}+h_K^{-1}\|v-\phi\|_{L^2(K)}\right)~,\\
\label{InterpBestApprox_L2}
\|v^{\partial K}-\cI^{\partial K}_K v\|_{L^2(K)}
&
\leq
  c\inf_{\phi\in V_p(K)}\left(h_K^{1+s}|v-\phi|_{H^{1+s}(K)}+h_K|v-\phi|_{H^1(K)}+\|v-\phi\|_{L^2(K)}\right)~.
\end{align}
\end{remark}

Combining Propositions~\ref{InterpolationBound1}
and~\ref{InterpolationBound2}, we obtain our key interpolation error
result,
\begin{theorem}\label{InterpolationError}
Suppose that $v\in H^{p+1}(K)$ for some $p\geq 1$. 
There is a scale-invariant constant $c=c(p,K)$ for which
\begin{align*}
|v-\cI_K v|_{H^1(K)}&\leq c h_K^{p} |v|_{H^{p+1}(K)}
\quad,\quad 
\|v-\cI_K v\|_{L^2(K)}\leq c h_K^{p+1} |v|_{H^{p+1}(K)}~.
\end{align*}
\end{theorem}

Once a proper notion of ``shape regularity'' is determined for
families of meshes $\{\cT_h\}$ consisting of curvilinear polygons, a
result such as
$|v-\cI v|_{H^1(\Omega)}\leq c h^p |v|_{H^{p+1}(\Omega)}$, where
$c=c(p)$ and $h=\max\{h_K: \,K\in \cT_h\}$, follows immediately.  A
meaningful analysis of how the constant $c(p,K)$ in
Theorem~\ref{InterpolationError} depends on $p$ and the geometric
features of $K$ is beyond the scope of the present work.  One might
expect measures such as a ``chunkiness parameter'' (a
natural generalization of aspect ratio, cf.~\cite[Defintion
4.2.16]{BrennerScott2002Book}), the number of edges, the curvature of
edges, and the length of edges with respect to the element diameter,
to play an important role in determining the dependence of $c=c(K)$ on
element geometry.  Indeed, the number of edges of $K$ clearly arises in
the proof of Proposition~\ref{InterpolationBound2}, when we bound
$|w_1|_{H^1(K)}$ using a sum of seminorms of functions associated the
vertices; see also Remark~\ref{H1_2NormEquivalence}, in which a sum over edges is used.
In contrast, the special analysis given for the L-shaped elements of
Example~\ref{LShape}, which have fixed size but increasing number of
edges as the mesh is ``refined'', provides an example in which neither
the number of edges nor their relation to the diameter of the element
have any bearing on the associated interpolation constant.
In each of the other examples in Section~\ref{Experiments}, the maximal
curvature of edges grows without bound as the diameters of the
elements shrink, and this has no apparent negative effect on the
convergence of the discretization error, which suggests that edge
curvature may not ultimately play such an important role in
interpolation error analysis either.  Additionally, some of the
families of meshes in Section~\ref{Experiments} consist entirely of
elements that are not star-shaped with respect to any ball, in which
case discussion of a chunkiness parameter is either meaningless, or
would have to take on a different form if it were to be applicable at all.
In summary, a thorough analysis of how local interpolation error
depends (or does not depend) on geometric features of elements is needed.
Further extensions of our interpolation error analysis of interest include:
\begin{enumerate} 
\item Estimates that
directly involve both the element diameter $h_K$ and a local
``polynomial degree'' $p_K$, in the manner of standard $hp$-finite
element analysis.  
\item Estimates that exploit the fact that $V_p(K)$ is a richer
  space than $\PP_p(K)$, often containing singular functions that may
  allow similar convergence results for interpolation 
under weaker regularity assumptions on $v$, as suggested by
Remarks~\ref{rem:LInfinityProp} and~\ref{InterpBestApprox}.
\end{enumerate}
We plan to pursue these extensions in subsequent work.

\section{Computing with Curved Trefftz Finite Elements}\label{Computation}
We recall that the functions that we wish to compute in $V_p(K)$ satisfy one
of two types of equations:
\begin{align}
\begin{cases}
\Delta v=f&\mbox{in } K\\
v=0&\mbox{on }\partial K
\end{cases}\quad,\quad
\begin{cases}
\Delta v=0&\mbox{in } K\\
v=g&\mbox{on }\partial K
\end{cases}~,
\end{align}
where $f\in\PP_{p-2}(K)$ and $g\in\PP_{p}(\partial K)$.  The first
type of equation is readily converted to the second type as follows.
Given $f\in\PP_{p-2}(K)$, one can explicitly construct a
$\hat{f}\in\PP_p(K)$ such that $\Delta \hat{f}=f$.  With such a
function in hand, the first type of problem is reduced to finding
$\hat{v}$ satisfying $\Delta\hat{v}=0$ in $K$ and $\hat{v}=-\hat{f}$
on $\partial K$.  Then $v=\hat{v}+\hat{f}$ satisfies the first
problem.  In~\cite[Theorem 2]{KarachikAntropova2010}, the authors show
that, if $p$ is a homogeneous polynomial of degree $j$, then the
polynomial $q$ of degree $j+2$ given by
\begin{align*}
 q(x)&=\sum_{k=0}^{[j/2]}
\frac{(-1)^k (j-k)!}{(j+1)!(k+1)!}\left(\frac{|x|^2}{4}\right)^{k+1}\,\Delta^k p(x)~,
\end{align*}
where $[j/2]$ denotes the integer part of $j/2$, satisfies
$\Delta q=p$.  Having reduced either type of problem to the
computation of a harmonic function with piecewise smooth boundary
data, we may now employ any number of boundary integral equation
techniques to compute such functions.  One such technique is to use Boundary Element
Methods for first-kind integral equations, as is done in BEM-FEM, to directly compute
the outward normal derivative $\partial v/\partial n$;  interior
point values are computed from layer potentials, as needed, for
quadrature approximation of the element stiffness matrix.
%The BEM-FEM approach has been
%to use Boundary Element Methods for first-kind integral equations for
%this purpose, providing a natural computation of the outward normal
%derivative $\partial v/\partial n$.  
The limitations in extending this
kind of approach to curved element boundaries in a natural way was one of the reasons
that we opted for Nystr\"om discretizations in~\cite{Anand2018}.  In
that work, we employed second-kind integral equations, which do not
directly yield $\partial v/\partial n$, but offered greater flexibility in
other areas that offset this downside.

Before describing the approach we use in the current work, we recall
the types of integrals we must compute in order to form the finite
element stiffness matrix.  They include integrals of the following
forms,
\begin{align*}
\int_K A\nabla v\cdot\nabla w\,dx \quad,\quad
  \int_{K}\mb{b}\cdot\nabla{v}w\,dx \quad,\quad \int_{K}c vw\,dx
  \quad,\quad \int_{K}f v\,dx\quad,\quad \int_e gv\,ds~,
\end{align*}
where $v,w\in V_p(K)$.
It is sometimes advantageous to use integration-by-parts on the first
of these integrals, yielding integrals of the forms
\begin{align*}
\int_K (\nabla\cdot A\nabla v) w\,dx\quad,\quad
\int_{\partial K} (A\nabla v\cdot\mb{n})w\,ds~.
\end{align*}
The benefits of such an approach become clear when $A$ is a constant
scalar on $K$, in which case the two integrals above simplify to 
\begin{align*}
A\int_K \Delta v\, w\,dx\quad,\quad
A\int_{\partial K} \frac{\partial v}{\partial n}\,w\,ds~.
\end{align*}
We note that $\Delta v\in\PP_{p-2}(K)$ and $w\in\PP_p(\partial K)$.
Further simplifications occur when $v\in V_p^{\partial K}(K)$ or $w\in V_p^{K}(K)
$---at least one of these two integrals vanishes.  We see then that,
in quadrature approximations of these kinds of integrals, we should
have access to function values and derivatives (up to second partials)
of functions in $V_p(K)$ in the interior of $K$, and
normal derivatives of such functions on $\partial K$---function values
and tangential derivatives of such functions on $\partial K$ are straightforward.

With these goals in mind, in~\cite{Ovall2018} we developed an approach
that delivers each of these quantities efficiently and with very high
accuracy, \textit{while performing all computations on the boundary}
$\partial K$.  The method is based on the fact that, on
simply-connected domains $K\subset\RR^2$, for each harmonic function $u$,
there is a family of \textit{harmonic conjugates} that differ from
each other only by additive constants.  We recall that $v$ is a
harmonic conjugate of $u$ on $K$ when $\Delta v=0$ in $K$ and $R\nabla
v=\nabla u$ in $K$, where the matrix $R$ rotates vectors clockwise by
$\pi/2$.  Such a pair of harmonic functions satisfy the Cauchy-Reimann
equations, and thus can be taken as the real and imaginary parts of a
complex analytic function in $K$.  More precisely, making the natural
identification between $z=x_1+\mathfrak{i} x_2\in\CC$ and
$x=(x_1,x_2)\in K$, the function defined by $w(z)=u(x)+\mathfrak{i}
v(x)$ is analytic in $K$.  Given both $u$ and $v$ on the boundary
$\partial K$, the value $w$ and its derivatives at points
inside $K$ can be obtained via Cauchy's integral formula,
\begin{align*}
w^{(k)}(z)=\frac{k!}{2\pi\mathfrak{i}}\oint_{\partial K}\frac{w(\xi)}{(\xi-z)^{k+1}}\,d\xi~,
\end{align*}
and the desired $k$th partial derivatives of $u$ (or $v$) can be extracted
from the real and imaginary parts of $w^{(k)}(z)$.  For example,
$w'(z)=u_{x_1}(x)-\mathfrak{i} u_{x_2}(x)$, where $u_{x_j}$ denotes
the partial derivative of $u$ in its $j$th argument.  For the second
partials, we have $w''(z)=u_{x_1x_1}(x)-\mathfrak{i} u_{x_1x_2}(x)$,
with $u_{x_2x_1}(x) =u_{x_1x_2}(x)$ and $u_{x_2x_2}(x) =-u_{x_1x_1}(x)$.
Furthermore, the
orthogonality of $\nabla u$ and $\nabla v$ in $K$ ensures the
following relationship between the normal and tangential derivatives
of $u$ and $v$ on $\partial K$.
\begin{align}\label{NormalTangential}
\frac{\partial u}{\partial n}=\frac{\partial v}{\partial t}\quad,\quad
\frac{\partial v}{\partial n}=-\frac{\partial u}{\partial t}~,
\end{align}
where $\partial v/\partial t$ denotes the tangential derivative along
$\partial K$ in the counter-clockwise direction.  With these
relationships, we see that it is possible to compute the normal
derivative of one harmonic function as the (much more convenient)
tangential derivative of a harmonic conjugate.   Therefore, this
general approach allows us to compute all of the quantities of
interest related to our harmonic function while performing all
computations on the boundary $\partial K$, as claimed.

Given the piecewise smooth boundary Dirichlet data $g$ of a harmonic function $u$,
any harmonic conjugate $v$ satisfies the complementary Neumann problem
\begin{align}\label{Neumann}
\Delta v=0\mbox{ in }K\quad,\quad \frac{\partial v}{\partial n}=-\frac{\partial g}{\partial t}~.
\end{align}
As stated earlier, solutions of~\eqref{Neumann} are only unique up to
additive constants, and we fix a particular member by specifying that
$\int_{\partial K}v\,ds=0$.  The trace of $v$ on $\partial K$ is
computed as the solution of the following second-kind integral
equation,
\begin{align}\label{IntegralEquation}
\frac{v(x)}{2}+\int_{\partial K}\left(\frac{\partial \Phi(x,y)}{\partial
  n(y)}+1\right)v(y)\,ds(y)=-\int_{\partial
  K}\Phi(x,y)\,\frac{\partial g}{\partial t}(y)\,ds(y)\mbox{ for }x\in\partial K~,
\end{align}
where $\Phi(x,y)=-(2\pi)^{-1}\ln|x-y|$ is the fundamental solution of
Laplace's equation.  The addition of $1$ to the integral kernel
$\partial \Phi(x,y)/\partial n(y)$ above, ensures
that~\eqref{IntegralEquation} is well-posed by enforcing that
$\int_{\partial K}v\,ds=0$.  Since $K$ will typically have corners,
the integral equation must be modified in their vicinity if we are to
understand the equation pointwise.  More specifically, if
$x_c\in\partial K$ is a corner point, and $x\in\partial K$ is not a
corner point, then
\begin{align}\label{IntegralEquation2}
|\partial K|v(x_c)+\frac{v(x)-v(x_c)}{2}+\int_{\partial K}\left(\frac{\partial \Phi(x,y)}{\partial
  n(y)}+1\right)(v(y)-v(x_c))\,ds(y)=-\int_{\partial
  K}\Phi(x,y)\,\frac{\partial g}{\partial t}(y)\,ds(y)~.
\end{align}
The case of multiple corners is handled similarly.
In the present work, as in~\cite{Ovall2018}, we
solve~\eqref{IntegralEquation} via a Nystr\"om discretization.   
Key to the practical success of this approach is the choice of 
quadrature schemes that are well-suited for the types of singularities
present in the integrands. The
interested reader may find the details in that paper~\cite{Ovall2018}.
Having computed $v$ on $\partial K$, we now have
access to the quantities of interest for $u$ as described above.

\begin{remark}[Multiply connected mesh cells]\label{NonSimplyConnected}
  If $K$ is not simply connected, the existence of harmonic conjugate
  pairs is not guaranteed, so the approach described above cannot be
  used.  In such cases, one can employ different integral equation
  techniques to efficiently and accurately compute the quantities
  necessary for assembling local stiffness matrices.  We mention the
  contribution~\cite{Greenbaum1993} (and references therein) in this
  regard.  The discussion in that work assumes smooth boundaries (at
  least $C^2$), so their approach must be modified in order to handle
  mesh cells having corners.  We intend to pursue this in subsequent
  work.
\end{remark}

\section{Numerical Experiments}\label{Experiments}
The experiments in this section illustrate the linear convergence rate
indicated by Theorem~\ref{InterpolationError} ($p=1$) for
$|u-\hat{u}|_{H^1(\Omega)}$ on simple model problems which nonetheless
illustrate the theoretical claims are achieved in practical
computations.  In the first set of experiments, three increasingly
complex families of meshes are used for the same problem on the unit
square.  In the second set of experiments, we explore the effects of
``nearly straight'' edges on convergence and conditioning for Type 1
and Type 2 elements.  For the final set of experiments, we consider a
problem for which the exact solution is known to have a singularity
due to a non-convex corner of the domain, and use it to illustrate the
approximation power of locally singular functions in our finite
element spaces.

\begin{example}[Three Curved Mesh Families]\label{ThreeMeshesExample}
Let $\Omega=(0,1)\times(0,1)$, and suppose that $u$ satisfies 
\begin{align*}
-\Delta u = 1\mbox{ in }\Omega\quad,\quad u=0\mbox{ on }\partial\Omega~.
\end{align*} 
Series representations of $u$ and $|u|_{H^1(\Omega)}^2$ are
\begin{align*}
%u&=\sum_{m,n=1}^\infty\frac{4(1-(-1)^m)(1-(-1)^n)}{mn(m^2+n^2)\pi^4}\,\sin(m\pi
%x)\,\sin(n\pi y)\\
u&=\sum_{m,n\in 2\NN-1}\frac{16}{mn(m^2+n^2)\pi^4}\,\sin(m\pi
x)\,\sin(n\pi y)~,\\
|u|_{H^1(\Omega)}^2&=\sum_{m,n\in 2\NN-1}\frac{64}{m^2n^2(m^2+n^2)\pi^6}\approx
3.51442537\times 10^{-2}~.
\end{align*}
We approximate $u$ by the finite element solution $\hat{u}\in V_1(\cT)$
on three different families of meshes, each indexed by a mesh
parameter $r$ that is inversely proportional to the characteristic
diameter of its mesh cells, see Figure~\ref{ThreeMeshes} for the case
$r=16$ of each, together with the corresponding computed finite
element solution $\hat{u}\in V_1(\cT)$.  Since the cells are of uniform
size, $r$ is the number of cells touching each edge of
$\partial\Omega$.

We refer to the first family of meshes as the Shuriken meshes, because
it consists of shuriken elements, as seen in Figure~\ref{ThreePolygons}),
which are naturally modified at the boundary to properly fit it.
There are three types of elements in this case, the corner elements,
edge elements and interior elements.  The second family is called the
Pegboard meshes, and it consists of two types of elements, the
half-washers and two-edge circles, as seen in
Figure~\ref{ThreePolygons}).  The third family is called the Jigsaw
meshes, and it has four different types of elements: corner pieces,
two different types of edge pieces, and interior pieces.
Of the nine different types of elements that appear in each of these
families, only the two-edge circles are convex.  In fact, none of the
other types of elements are even star-shaped.  Furthermore, each of
the jigsaw elements have at least two non-convex corners, which
implies that the local space $V_1(K)$ for such an element will include
functions that are singular, i.e. not in $H^2(K)$.

\begin{figure}
\begin{center}
\begin{tabular}{cc}
\includegraphics[width=2.2in]{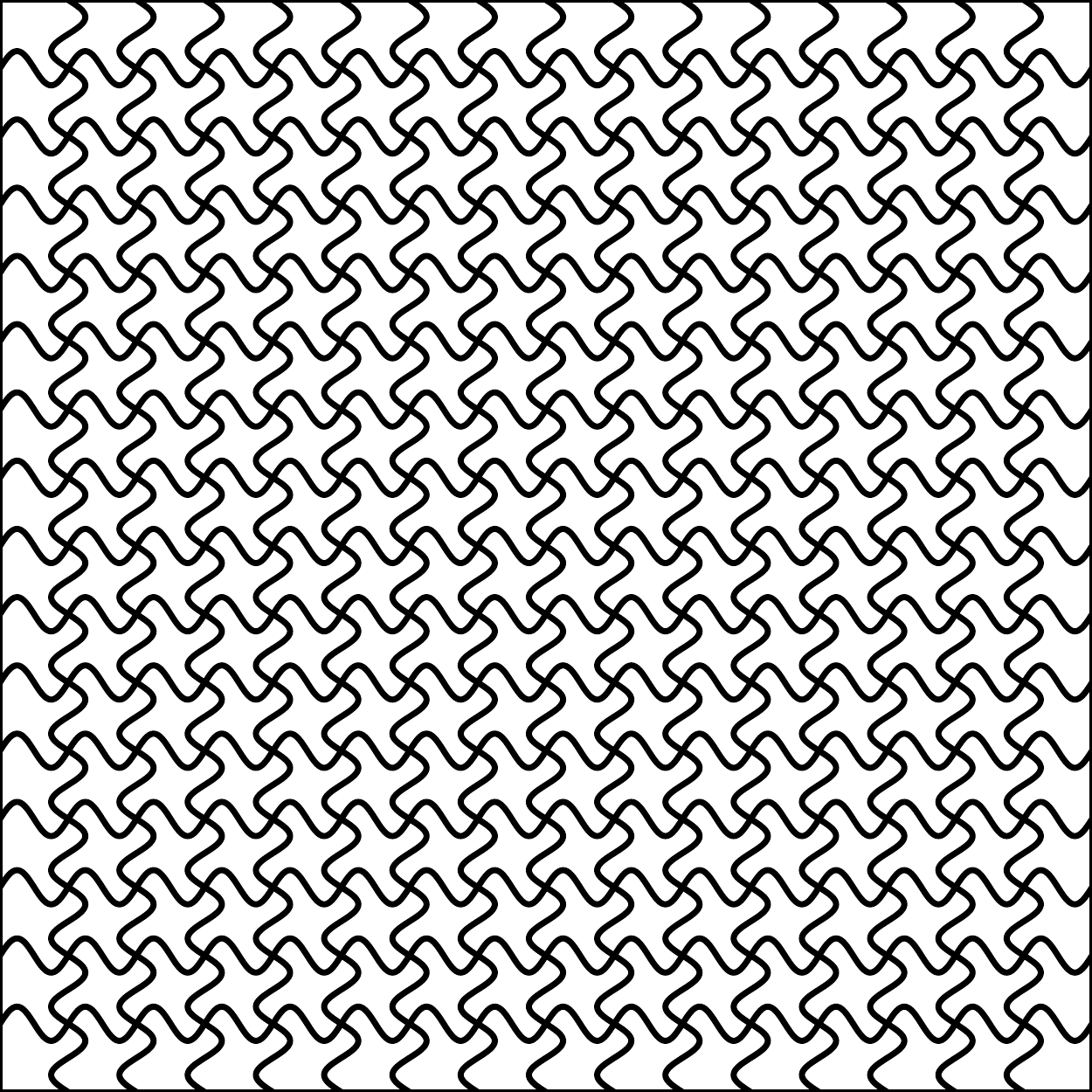}&
\includegraphics[width=3.2in]{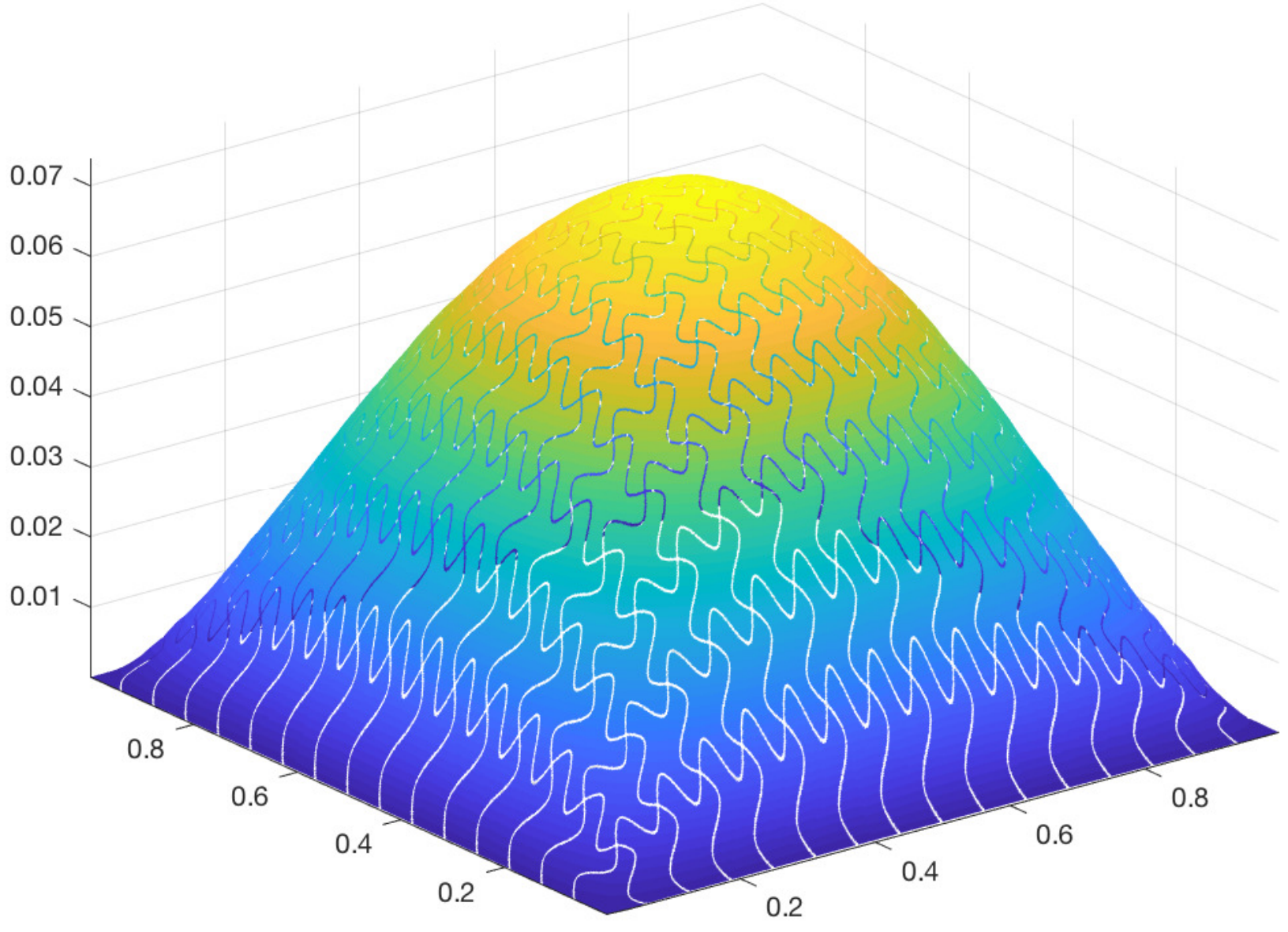}\\
\includegraphics[width=2.2in]{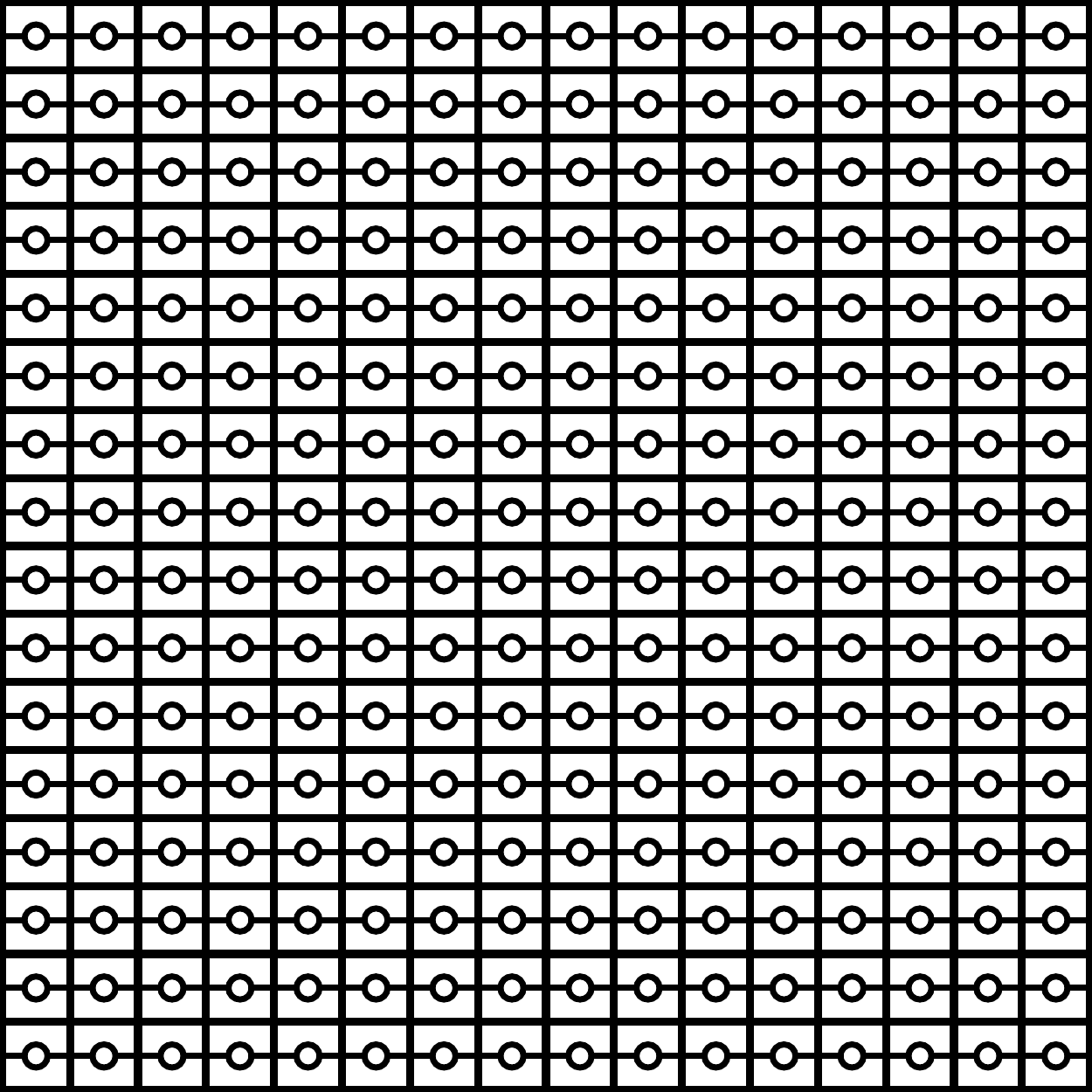}&
\includegraphics[width=3.2in]{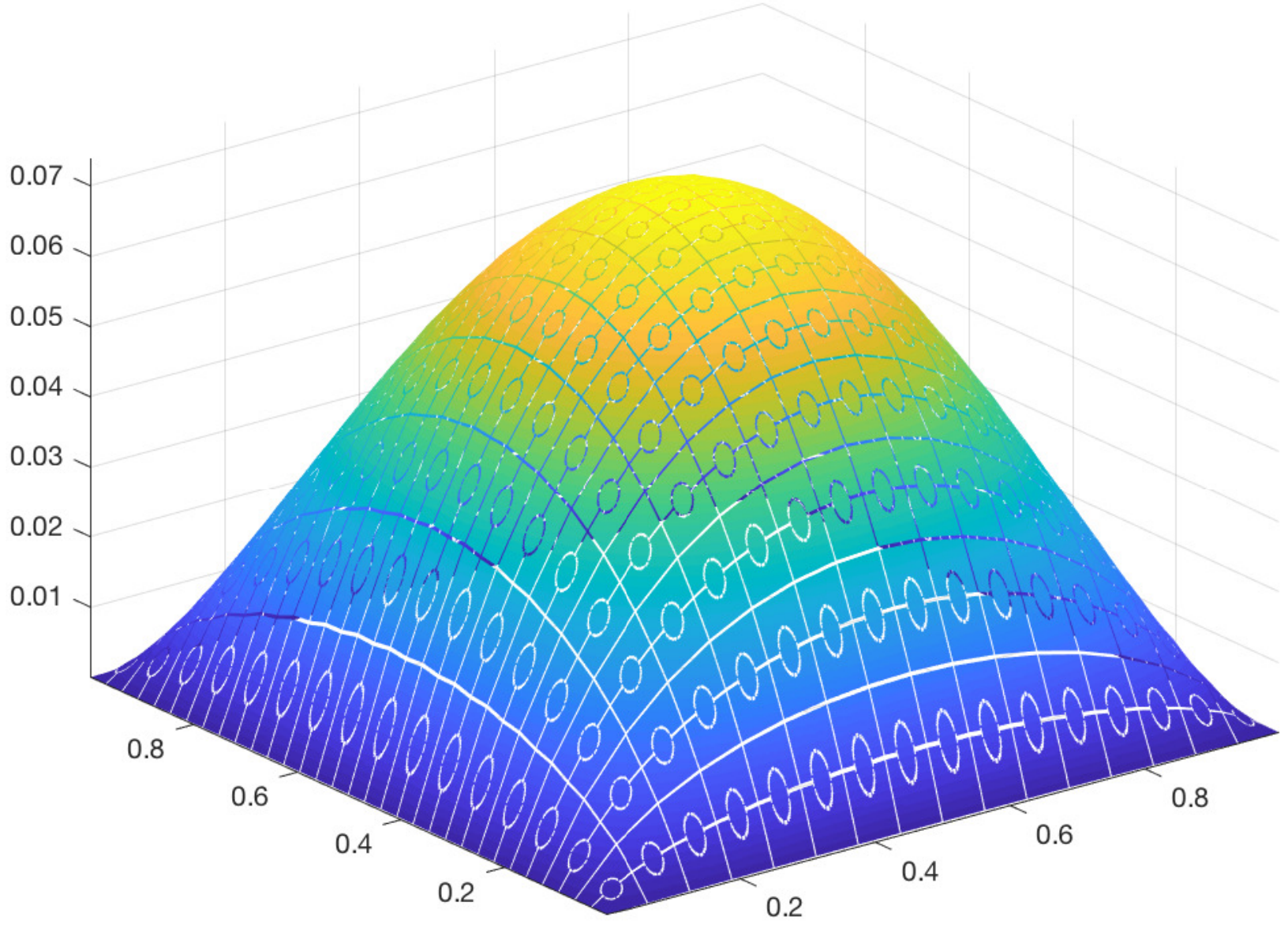}\\
\includegraphics[width=2.2in]{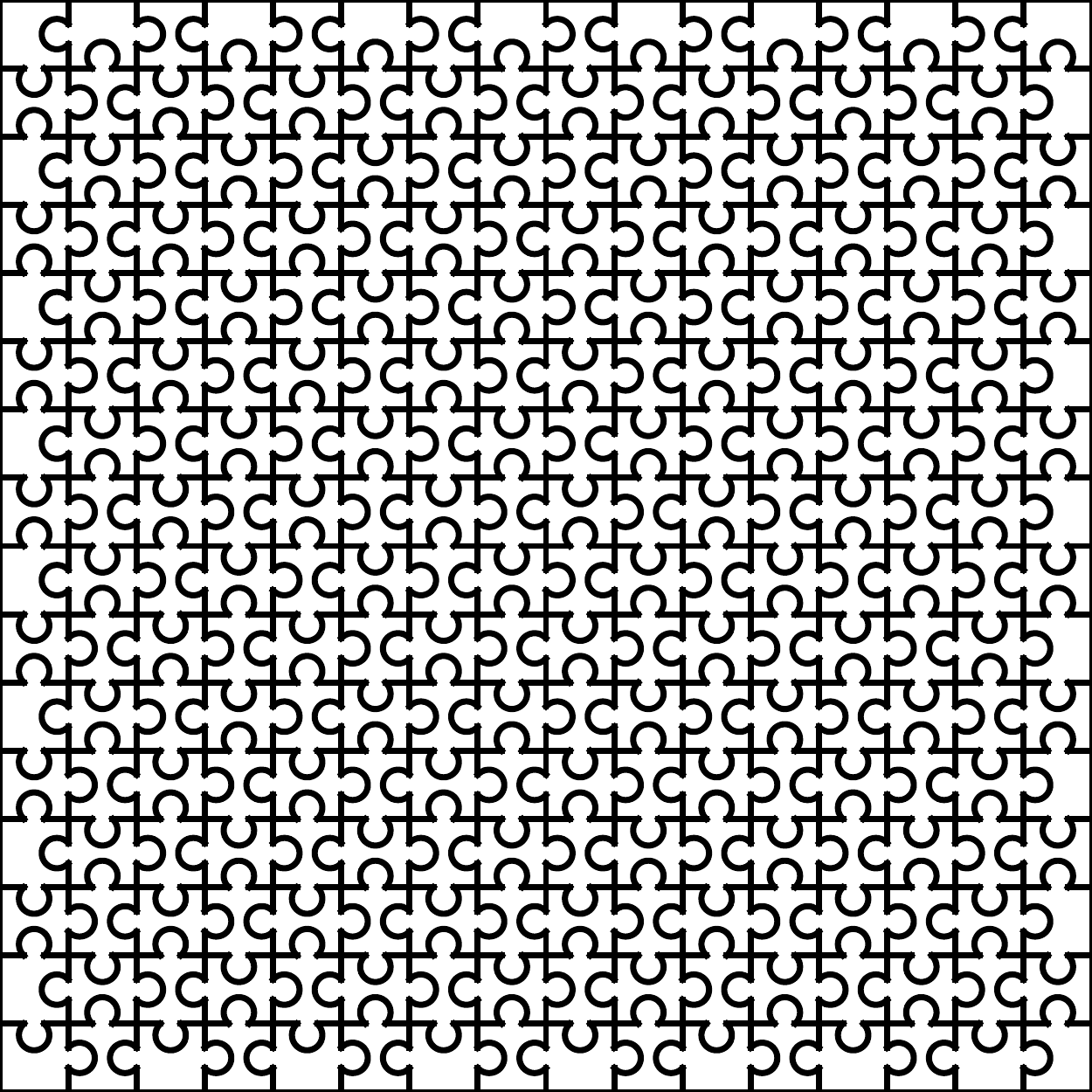}&
\includegraphics[width=3.3in]{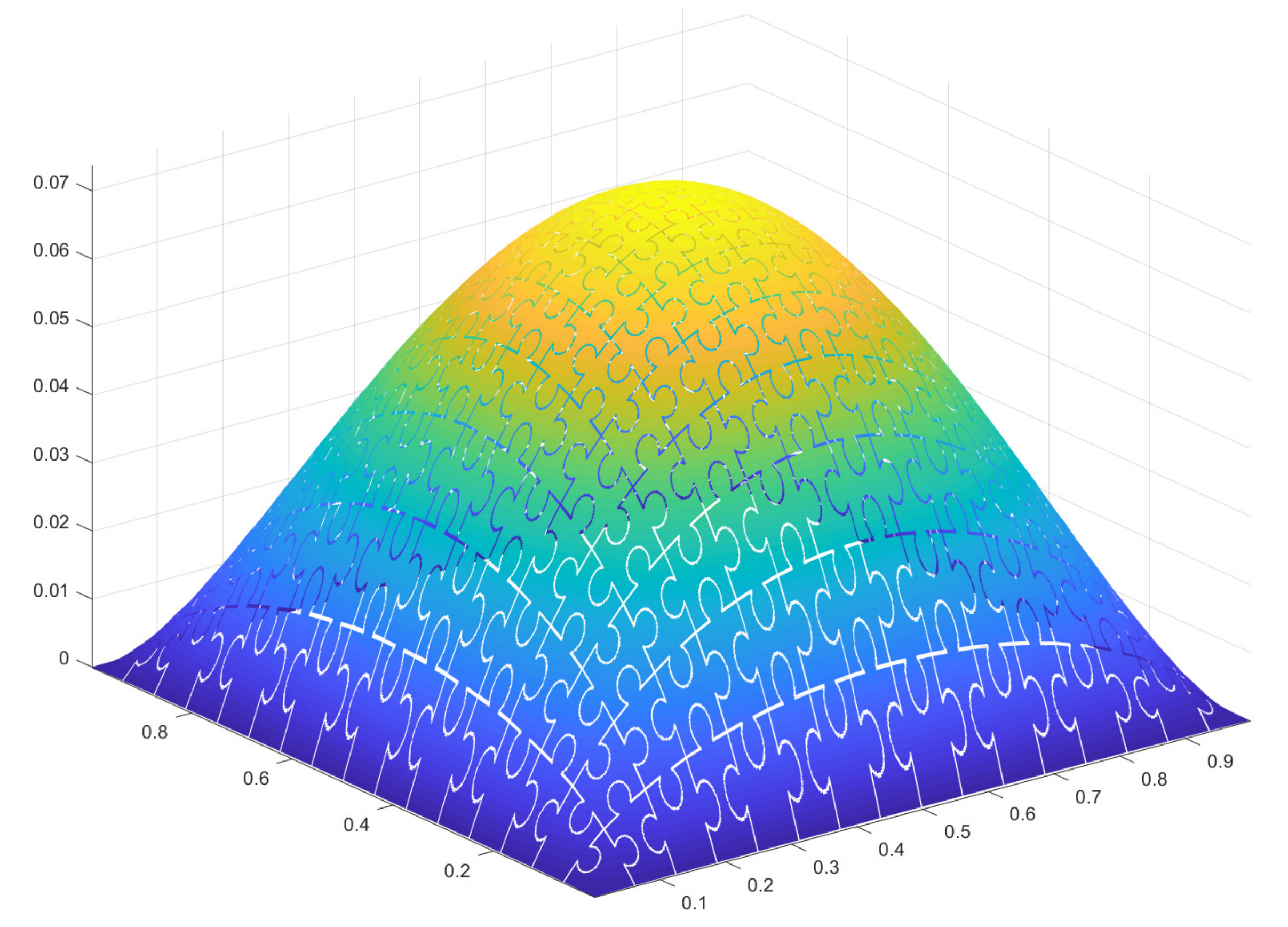}
\end{tabular}
\end{center}
\caption{\label{ThreeMeshes} The Shuriken, Pegboard and Jigsaw meshes
  for $r=16$, as well as the computed finite element solution
  $\hat{u}\in V_1(\cT)$ on these meshes.}
\end{figure}

The discretization error $|u-\hat{u}|_{H^1(\Omega)}$ satisfies
$|u-\hat{u}|_{H^1(\Omega)}^2=|u|_{H^1(\Omega)}^2-|\hat{u}|_{H^1(\Omega)}^2$, making it
straight-forward to compute once $\hat{u}\in V_1(\cT)$ has been
computed.  These errors, and ratios of consecutive errors, are given
in Table~\ref{ThreeMeshConvergence} for each of the families, demonstrating the expected linear
convergence.  As an interesting comparison, we also include the errors
and ratios for the discretizations that would arise had we chosen the
Type 1 definition of $\PP_1(e)$ for the shuriken elements.  This choice leads to local spaces
having dimension $4$ for mesh cells not touching the boundary, in
contrast to the dimension $8$ local spaces using the Type 2 definition
of $\PP_1(e)$.  We recall that $\PP_1(K)\not\subset V_1(K)$ for Type 1
elements (unless $K$ is a straight-edge polygon), and we see that there is no
convergence at all in $H^1$ in this case!  In all three cases, optimal
order convergence, $|u-\hat{u}|_{H^1(\Omega)}=\cO(h^2)$, is obtained
when Type 2 elements are used.  
%We remark that the presence of singular functions in the bases for the jigsaw meshes does not hurt
%the convergence.

\begin{table}
\caption{\label{ThreeMeshConvergence} Discretization
errors, $|u-\hat{u}|_1$, and error ratios for the Shuriken, Pegboard,
and Jigsaw meshes.  For the Shuriken meshes, both Type 1 and Type 2 
elements are used.  For the other meshes, only Type 2 elements are used.}
\begin{center}
\begin{tabular}{|c||cc|cc|cc|cc|}\hline
  &\multicolumn{2}{c|}{Shuriken, Type 1}&\multicolumn{2}{c|}{Shuriken,
                                          Type 2}&\multicolumn{2}{c|}{Pegboard}&\multicolumn{2}{c|}{Jigsaw}\\
$r$&$|u-\hat{u}|_{H^1(\Omega)}$&ratio&$|u-\hat{u}|_{H^1(\Omega)}$&ratio&$|u-\hat{u}|_{H^1(\Omega)}$&ratio&$|u-\hat{u}|_{H^1(\Omega)}$&ratio\\\hline
4 &7.882e-02&& 5.629e-02& &3.470e-02&         &3.209e-02&\\
8&8.200e-02&0.961&2.847e-02&1.977&1.726e-02&2.010&1.538e-02&2.086\\
16&8.813e-02&0.930&1.429e-02&1.993&8.537e-03&2.022&7.559e-03&2.035\\
32&9.232e-02&0.955&7.150e-03&1.998&4.233e-03&2.017&3.754e-03&2.014\\
64&9.470e-02&0.975&3.576e-03&1.999&2.106e-03&2.010&1.871e-03&2.006\\
128&9.600e-02&0.987&1.788e-03&2.000&1.050e-03&2.006&9.327e-04&2.006\\
256&9.662e-02&0.993&8.937e-04&2.001&5.239e-04&2.004&4.628e-04&2.015\\
512&9.694e-02&0.997&4.463e-04&2.003&2.611e-04&2.007&2.247e-04&2.060\\\hline
\end{tabular}
\end{center}
\end{table}

\end{example}

\begin{example}[Perturbed Triangle Mesh]
  Let $\Omega$ and $u$ be as in the previous example.  We again
  approximate $u$ by its finite element solution $\hat u\in V_1(\cT)$,
  where the mesh $\cT$ consists perturbed triangular elements, as
  shown in Figure \ref{CurvyTriangleMesh}, each of which has one
  curved edge.  Reference elements, one convex and the other
  non-convex, are obtained by splitting the unit square
  $(0,1)\times(0,1)$ using a circular arc through the vertices $(1,0)$
  and $(0,1)$ whose center is $(-a,-a)$, for some $a>0$.  The radius
  of curvature of this curved edge is $\sqrt{a^2+(1+a)^2}$, and it 
  approaches a straight line as $a$ increases.  More specifically, the
  maximum distance between a point on the curved edge and the closest
  point to it on the line between $(1,0)$
  and $(0,1)$ is $(\sqrt{(2a+1)^2+1}-(2a+1))/\sqrt{2}$, which behaves
  like $1/(4a\sqrt{2})$ as $a\to\infty$.  For the corresponding finite element
  meshes, these reference element pairs are scaled so that their
  straight edges have length $1/r$.
\begin{figure}
\includegraphics[width=0.23\textwidth]{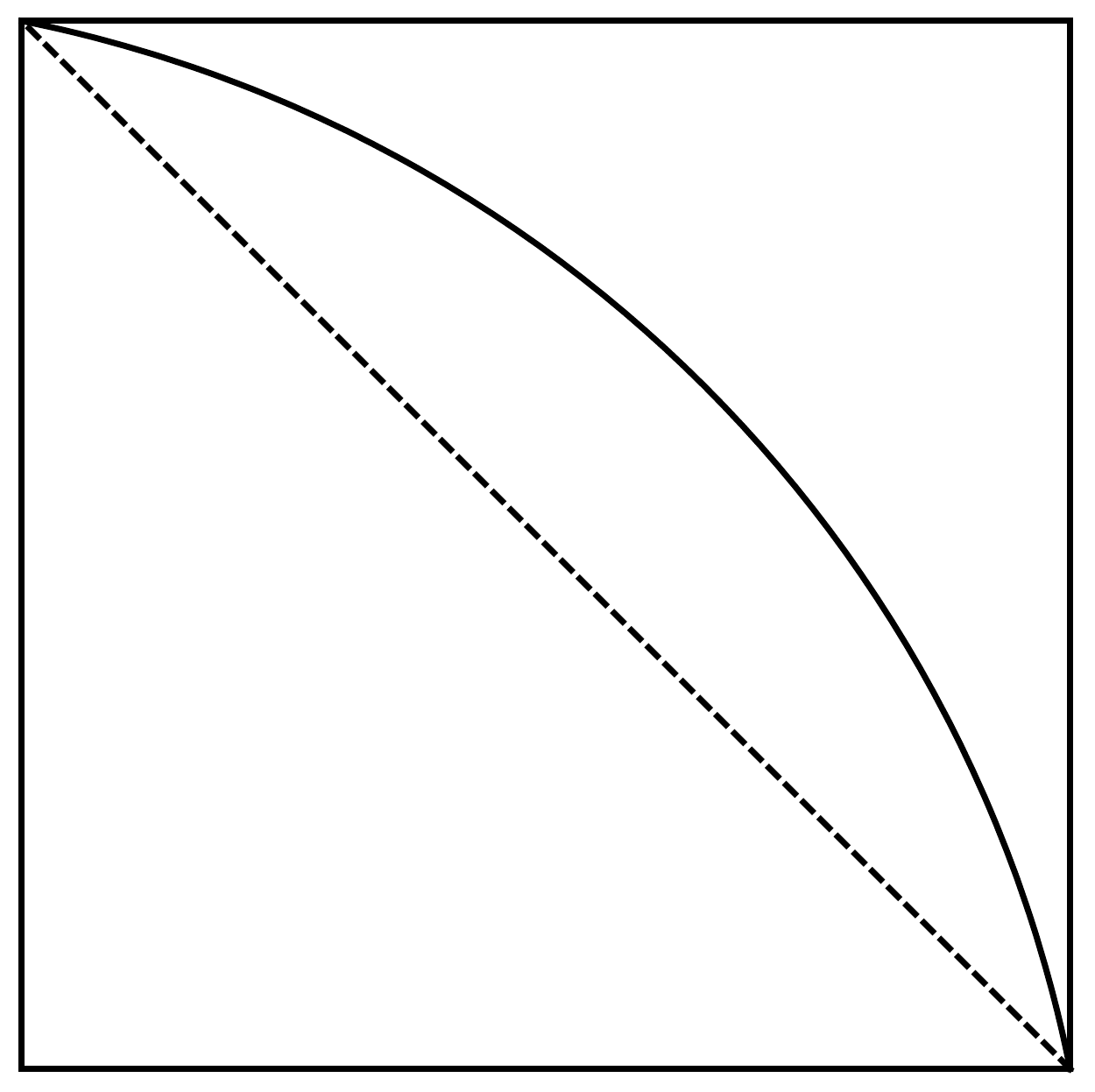}
\includegraphics[width=0.23\textwidth]{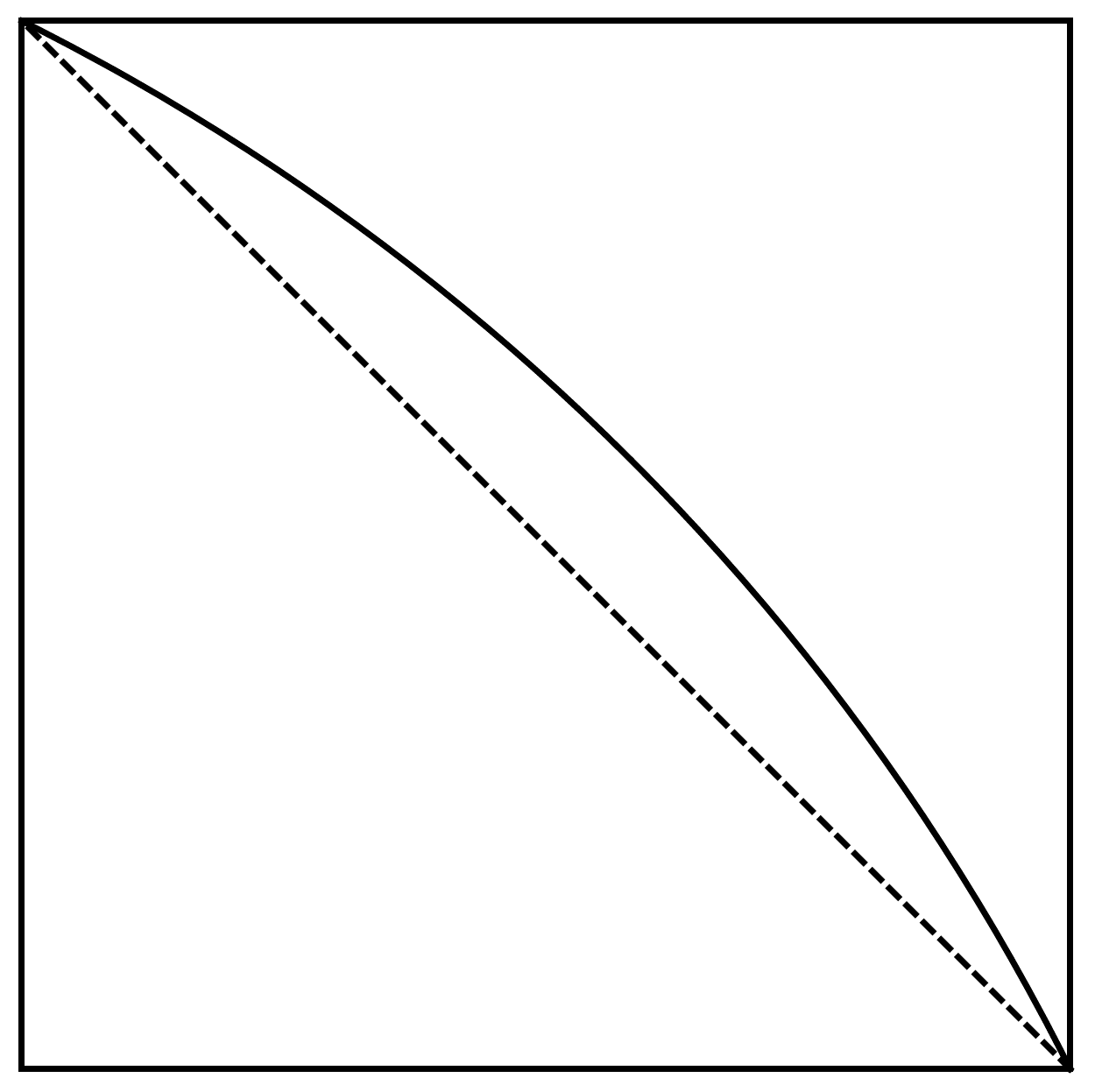}
\includegraphics[width=0.23\textwidth]{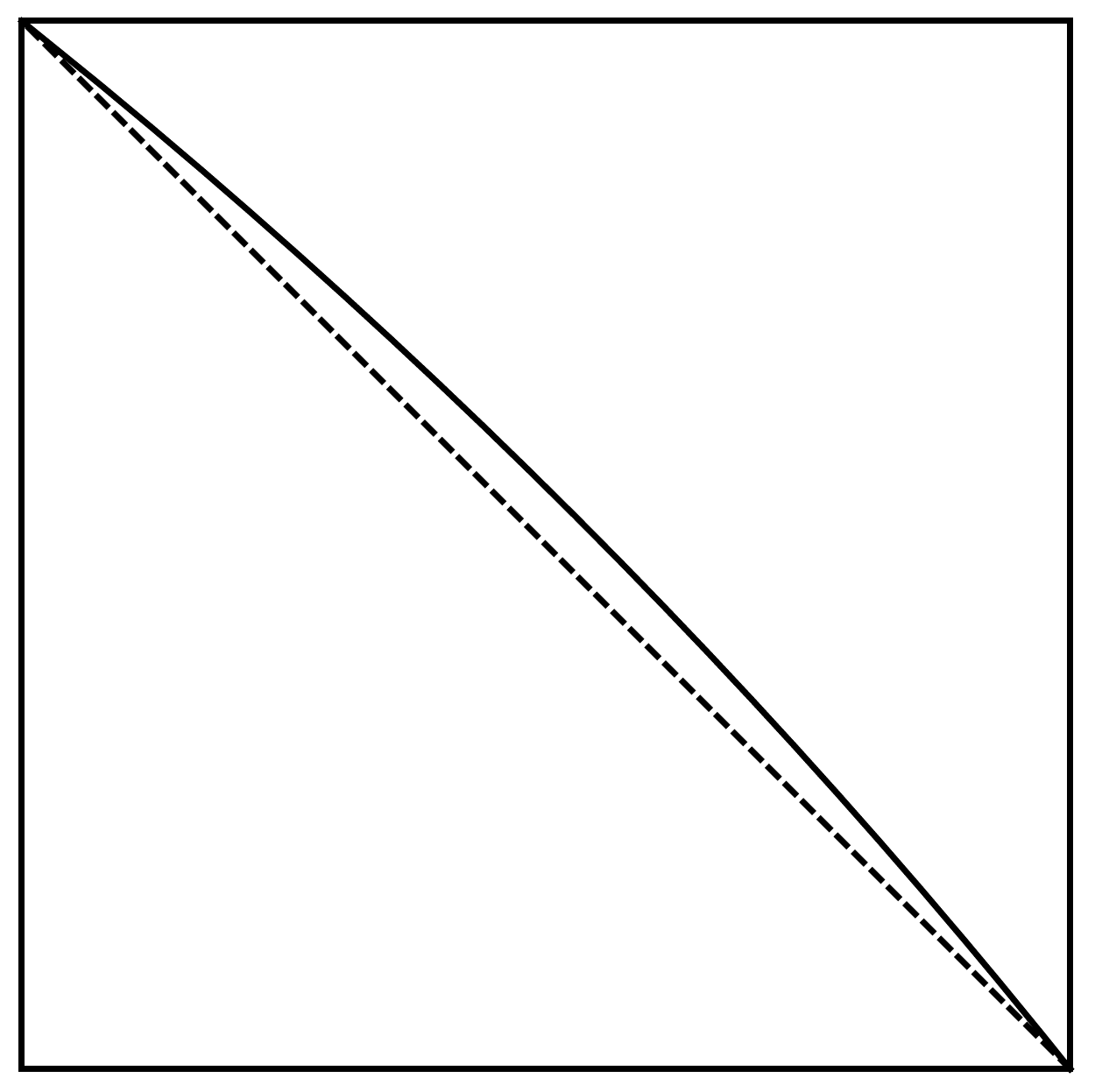}
\includegraphics[width=0.23\textwidth]{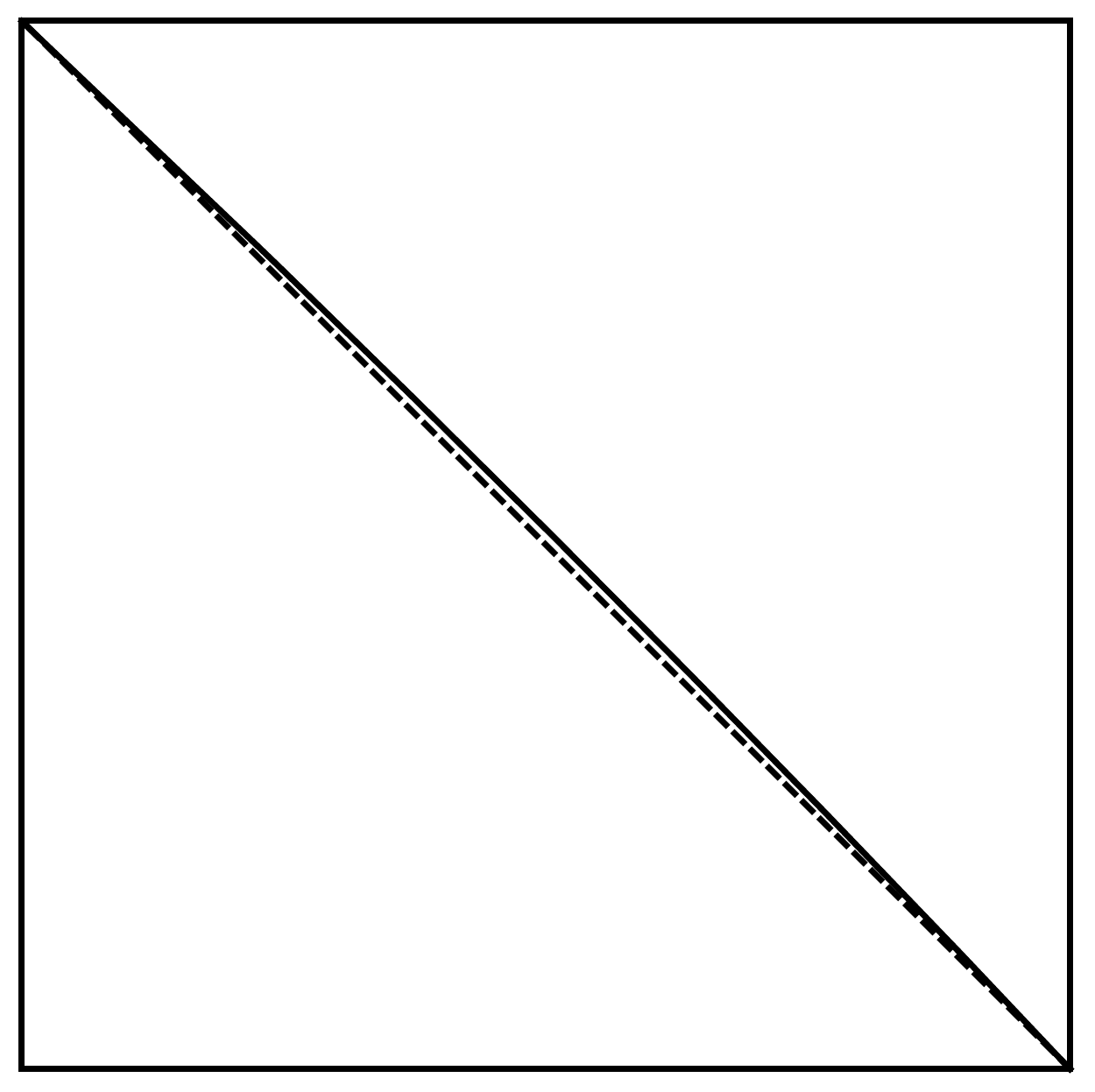}\\
\includegraphics[width=0.23\textwidth]{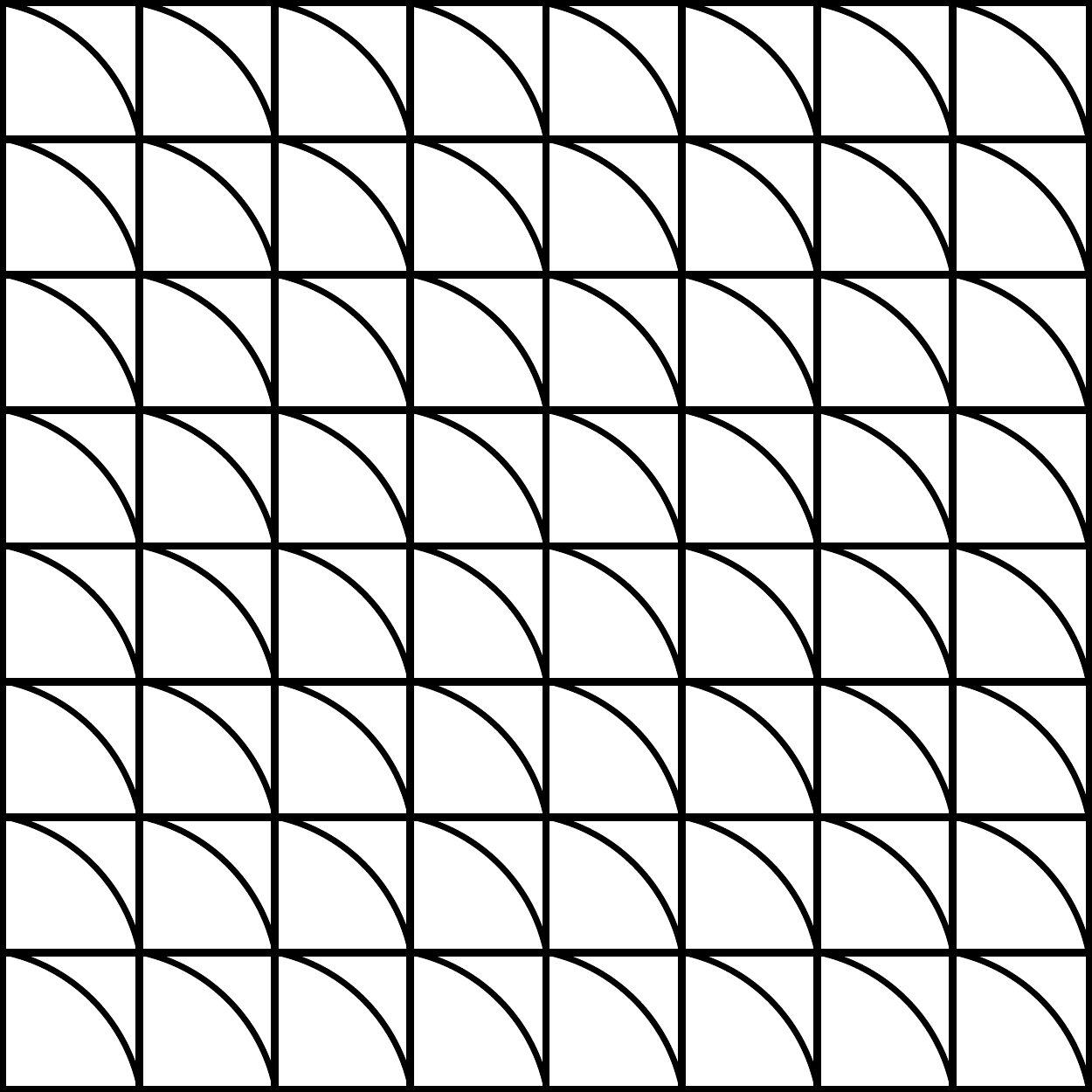}
\includegraphics[width=0.23\textwidth]{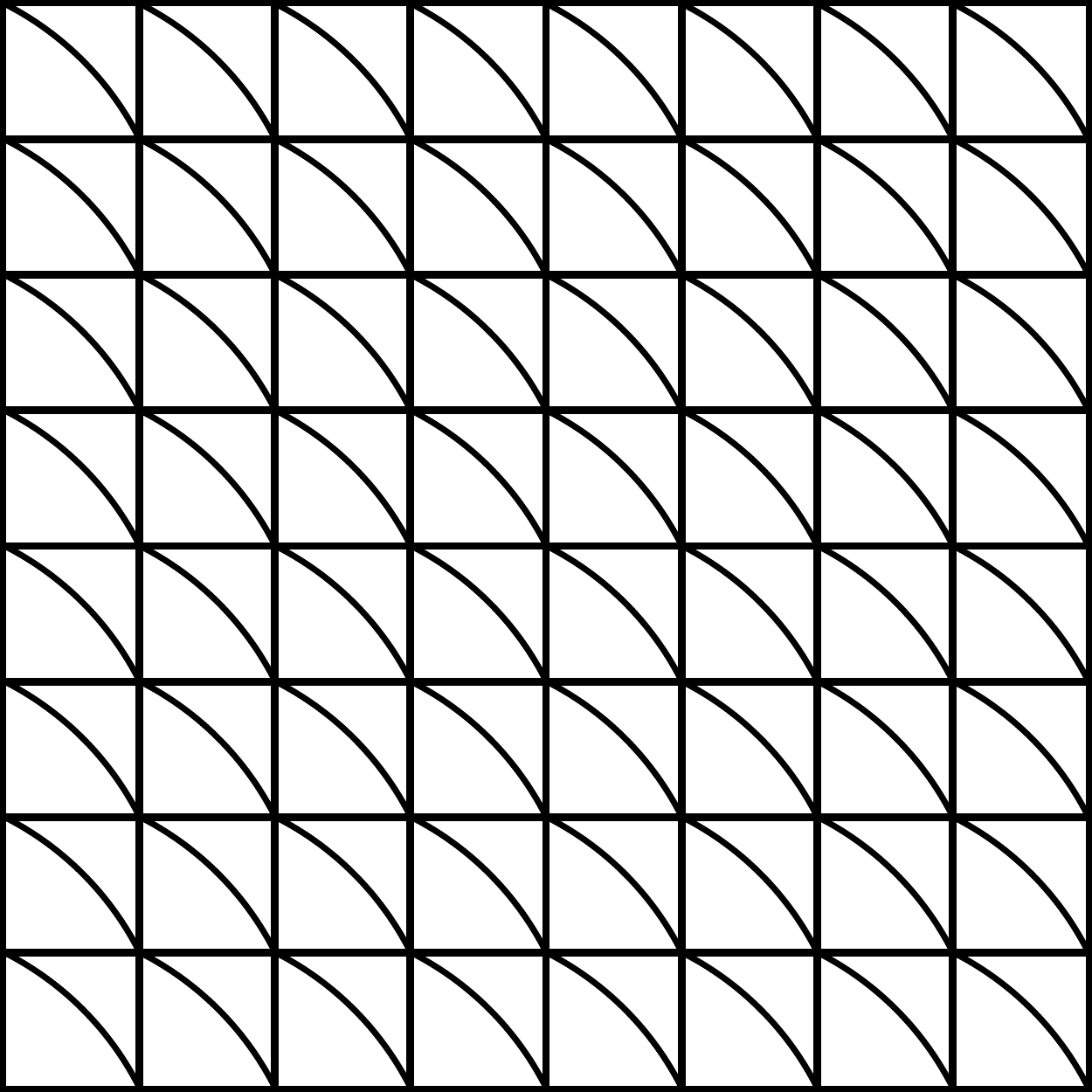}
\includegraphics[width=0.23\textwidth]{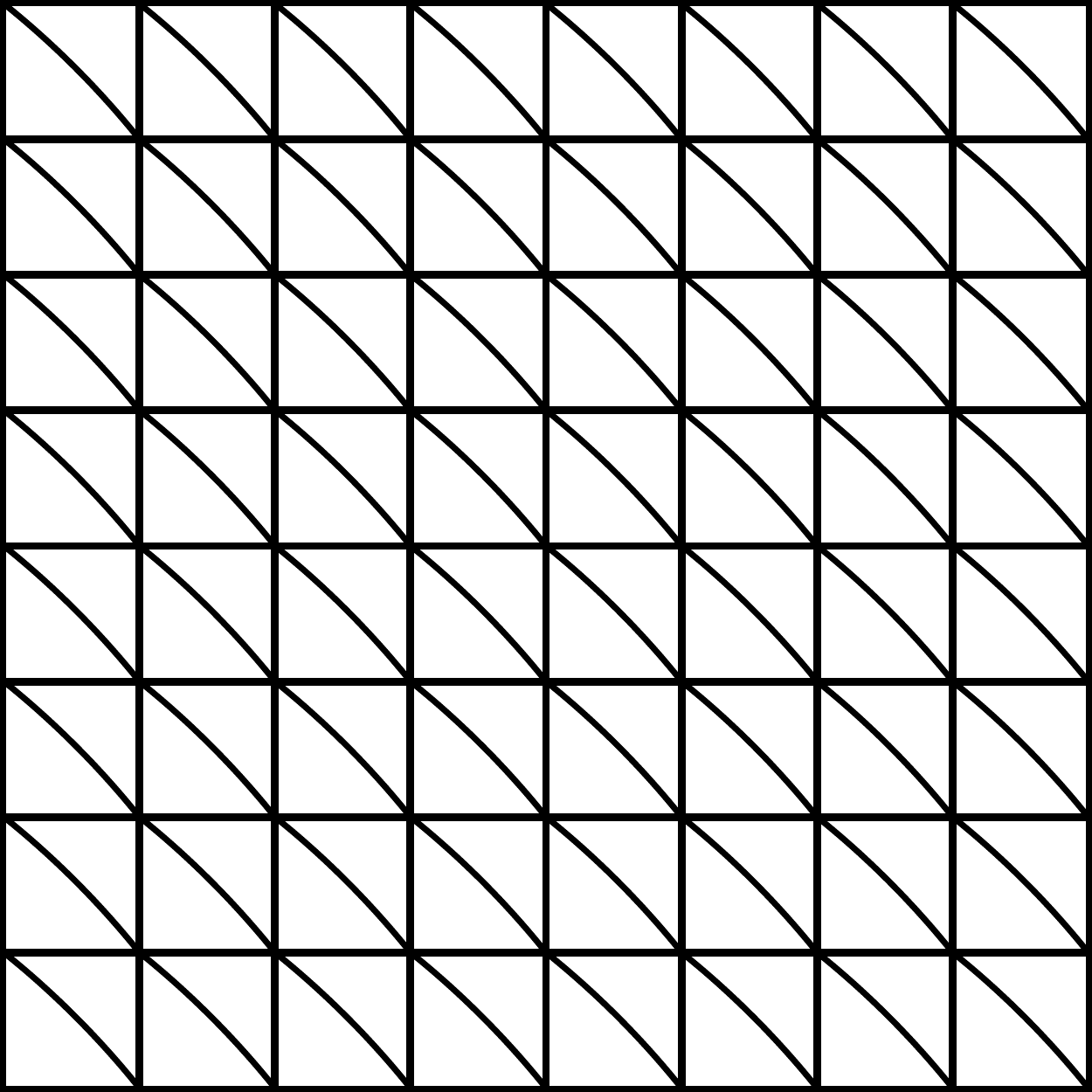}
\includegraphics[width=0.23\textwidth]{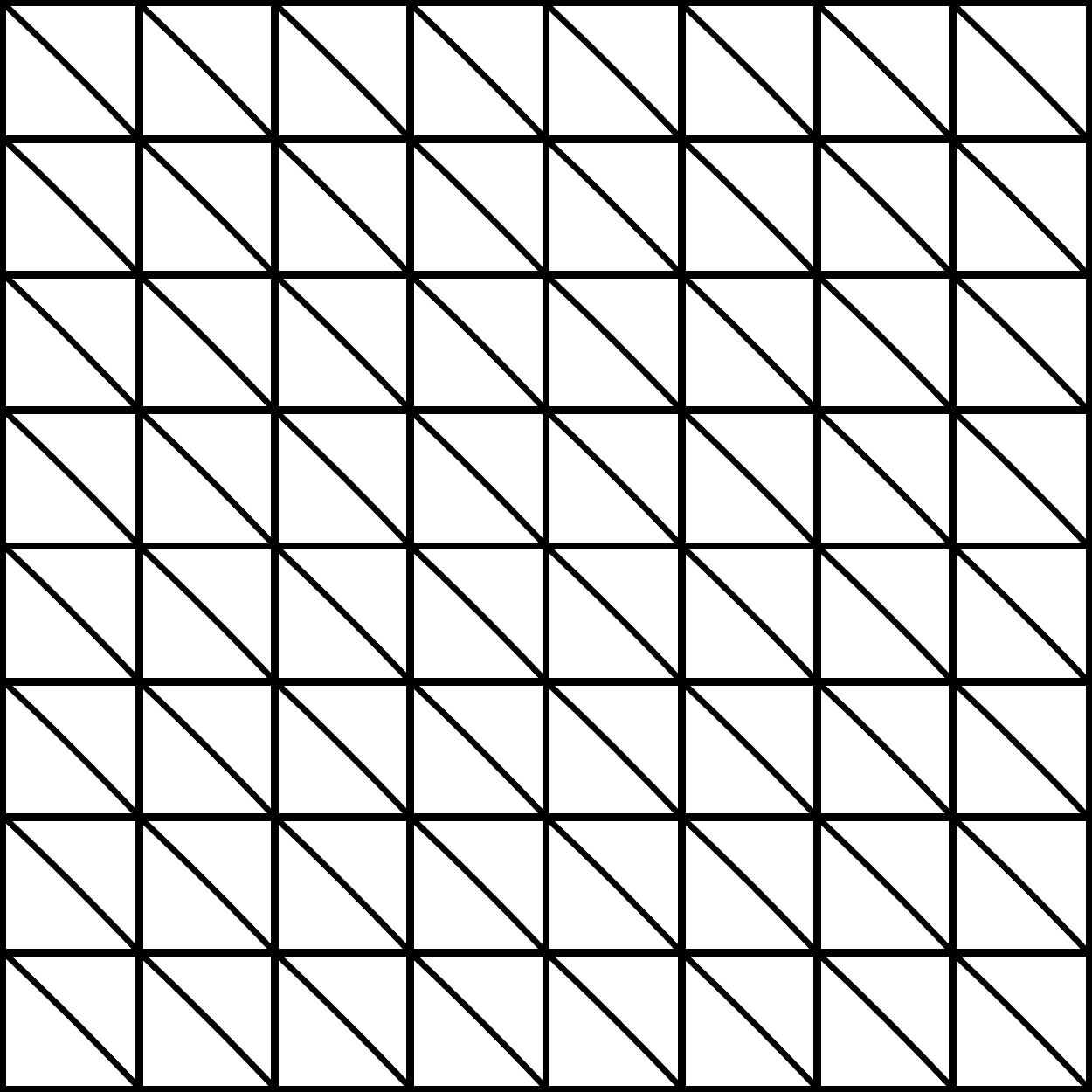}
\caption{\label{CurvyTriangleMesh} Reference elements (top) for the
  Perturbed Triangle meshes (bottom), shown with $r=8$, for
  $a=1/4,1,4,16$ from left to right. The dashed lines are straight ($a=\infty$).}
\end{figure}

In Figure~\ref{CurvyTriangleConvergence}, we report both the errors
$|u-\hat u|_{H^1(\Omega)}$ and the spectral condition numbers
$\kappa_2(A)$ of the stiffness matrices for Type 1 and Type 2 elements
as the meshes is refined, for several values of $a$.  For both types
of elements, $A$ is diagonally rescaled,
$a_{ij}\longleftarrow a_{ij}/\sqrt{a_{ii}a_{jj}}$, before computing
the condition numbers.  As expected, the Type 2 elements exhibit
optimal order convergence throughout the refinements.  For Type 1
elements, the convergence curves improve as the $a$ is increased, in
the sense that they stay roughly parallel to their Type 2 counterparts
through more levels of refinement, but the convergence curves for Type
1 elements eventually level off, indicating a threshold beyond which
the error does not decrease.  The condition number plots for the Type
1 and Type 2 elements provide a complementary comparison, for which
the Type 2 elements yield condition numbers that are
\textit{eventually} close to, and grow at the same rate as, those of
the Type 1 elements, but may be significantly larger than their Type 1
counterparts for coarser meshes when the curved edges are nearly
straight.  The condition numbers for Type 1 elements grow like $r^2$
as the mesh is refined (i.e. like $\dim(V)$), which is accordance with
standard linear (bilinear) elements on triangular (rectangular)
meshes.  We observe, based on computations done for
$a=1/4,1,4,\ldots,4096$, that the condition numbers on the coarsest
meshes ($r=4$) for Type 2 elements appear to grow quadratically in
$a$.  We also observe an apparent correlation between when the
convergence curves for Type 1 elements tend to level off and when the
condition numbers for Type 2 elements transition from a relatively
flat phase to behaving like their Type 1 counterparts.  It is a topic
of future investigation to better understand how element shapes and
``polynomial orders'' $p$ affect convergence and conditioning for both
types of elements.
\begin{figure}
\includegraphics[width=0.45\textwidth]{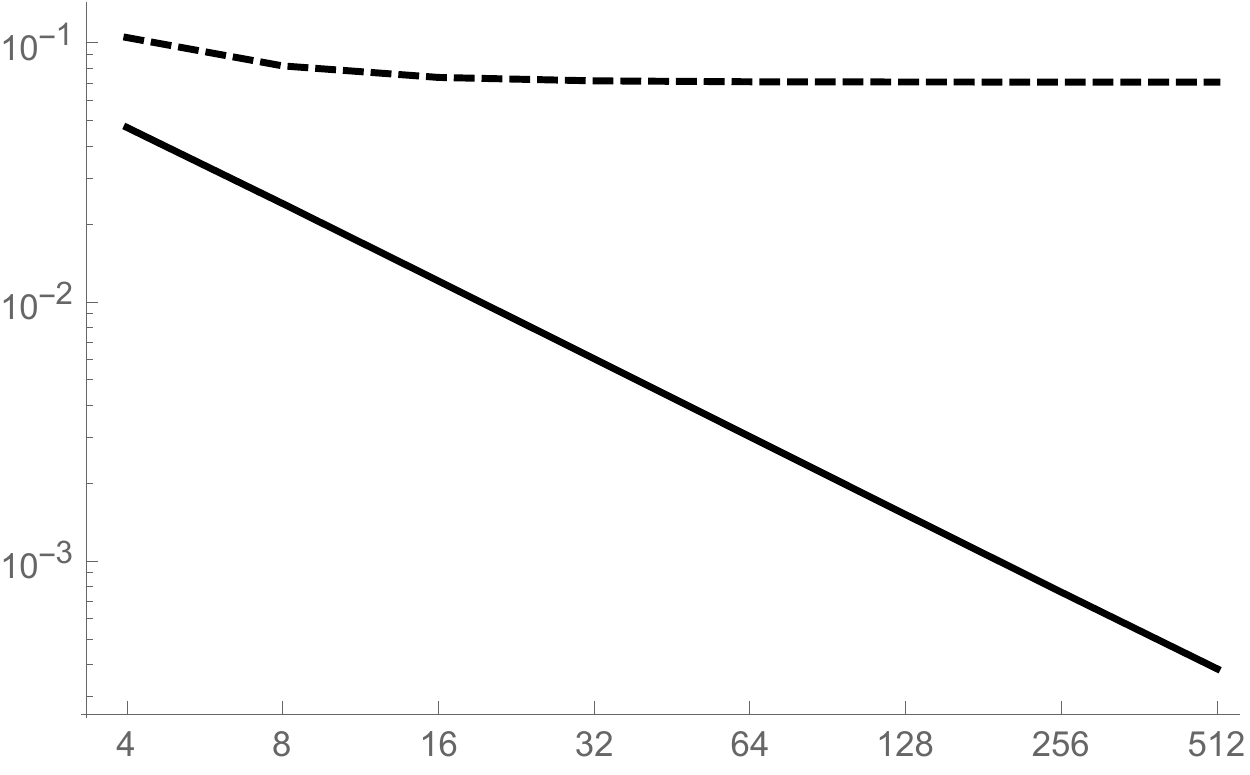}
\includegraphics[width=0.45\textwidth]{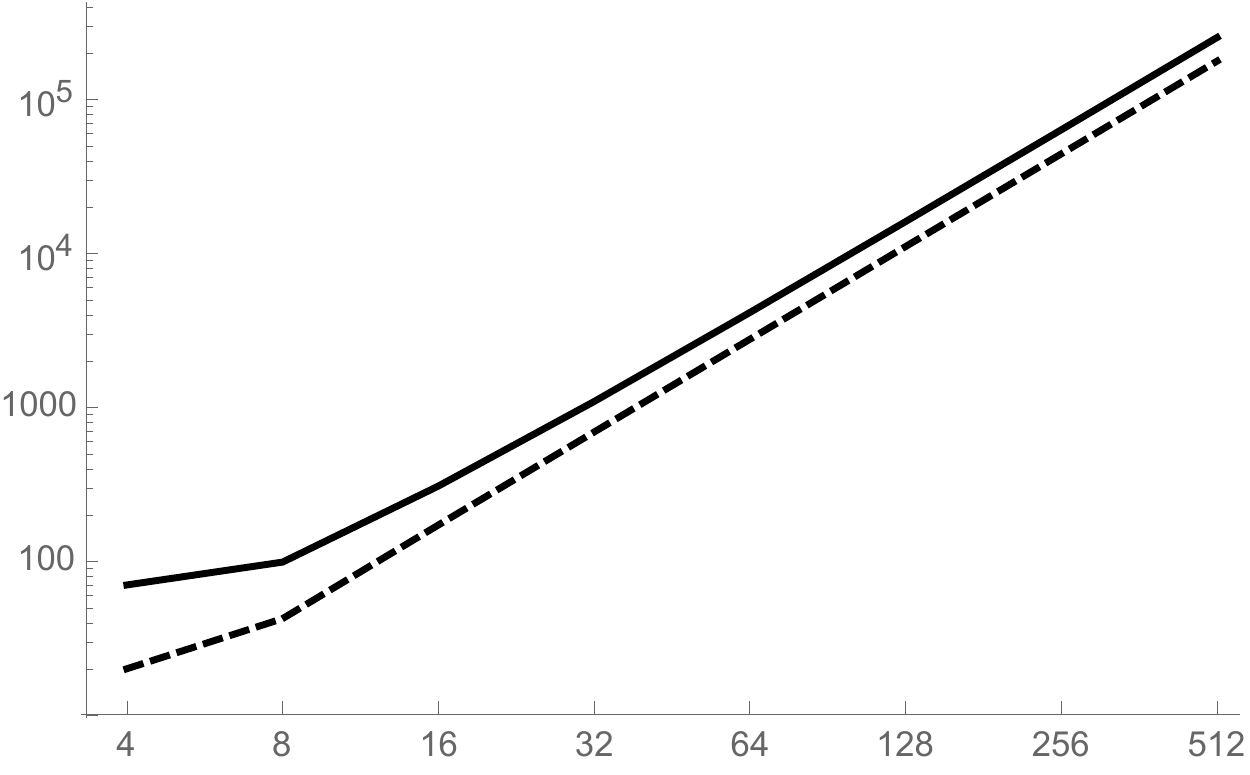}
\includegraphics[width=0.45\textwidth]{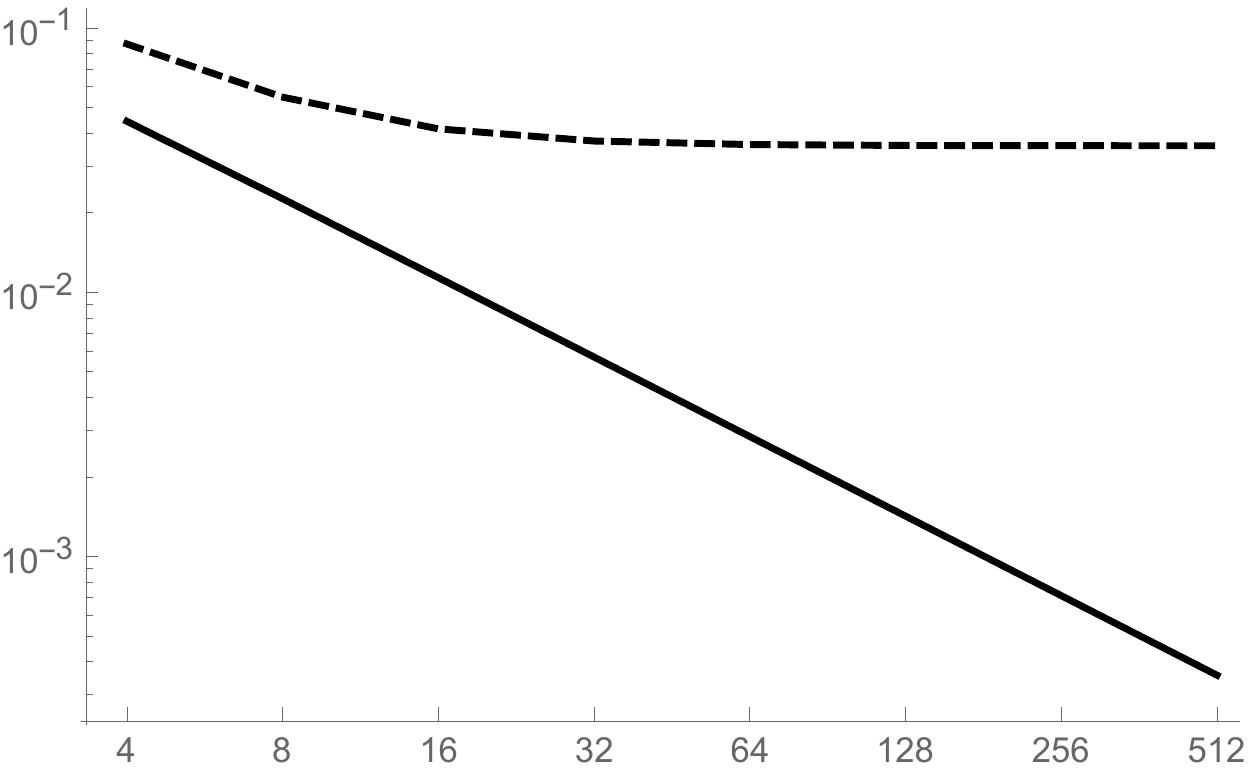}
\includegraphics[width=0.45\textwidth]{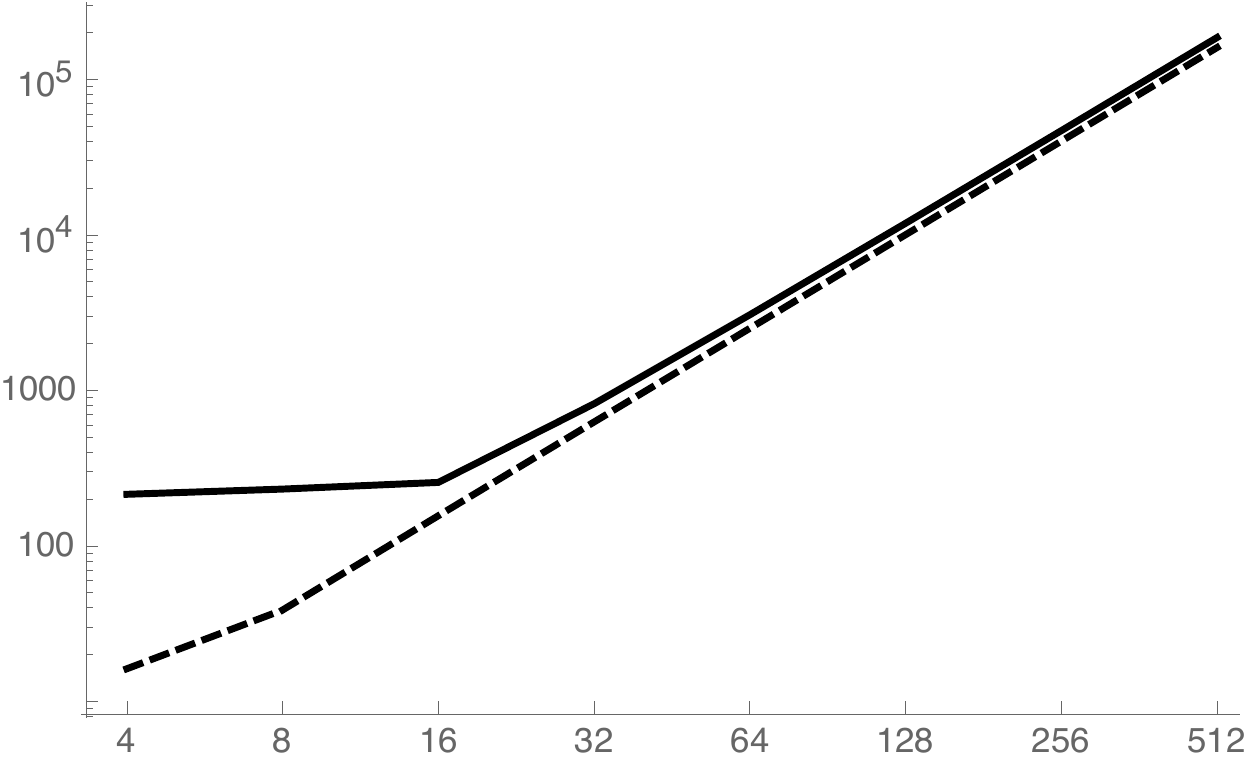}
\includegraphics[width=0.45\textwidth]{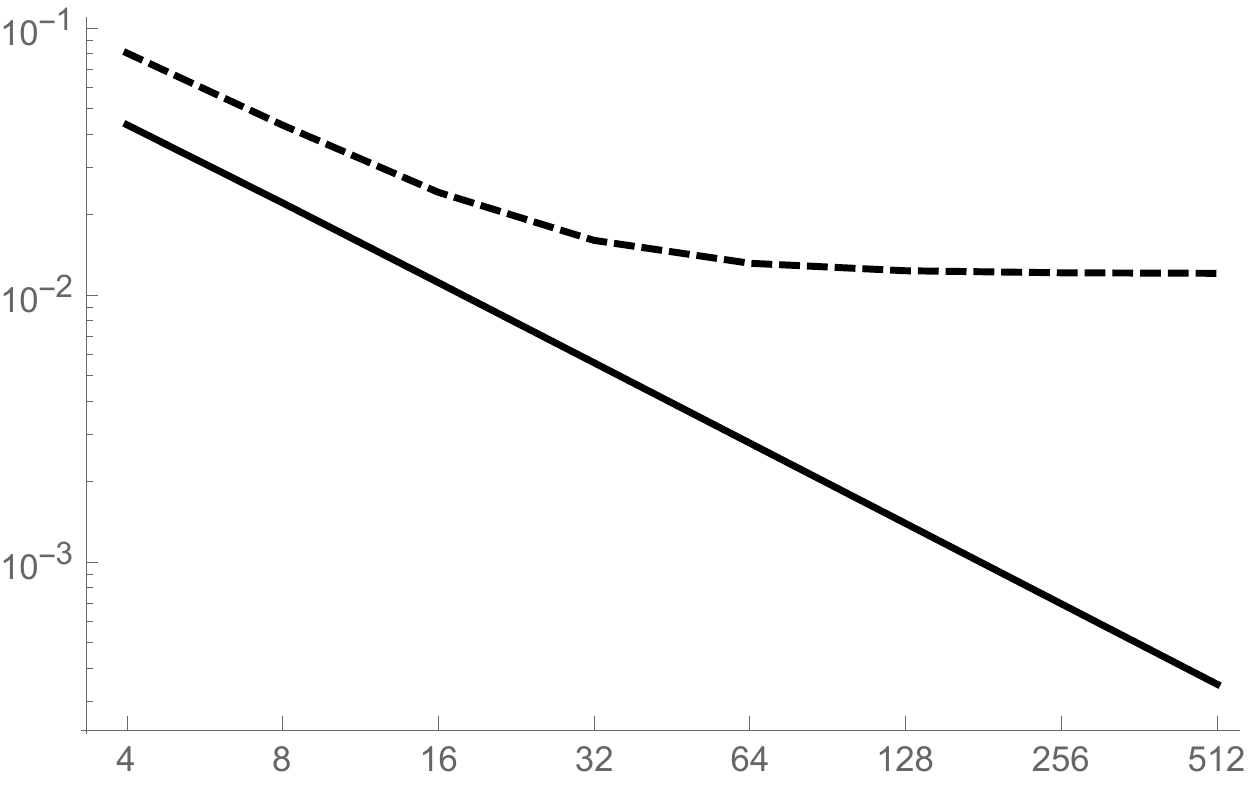}
\includegraphics[width=0.45\textwidth]{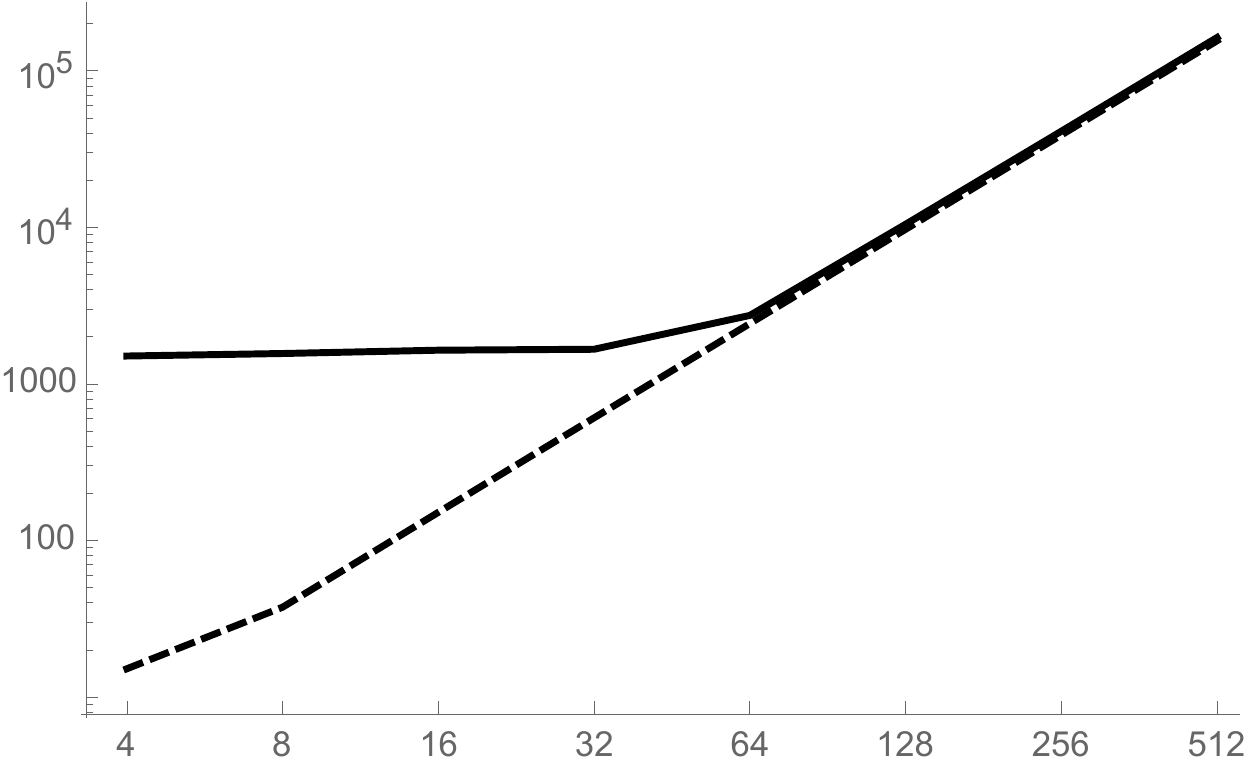}
\includegraphics[width=0.45\textwidth]{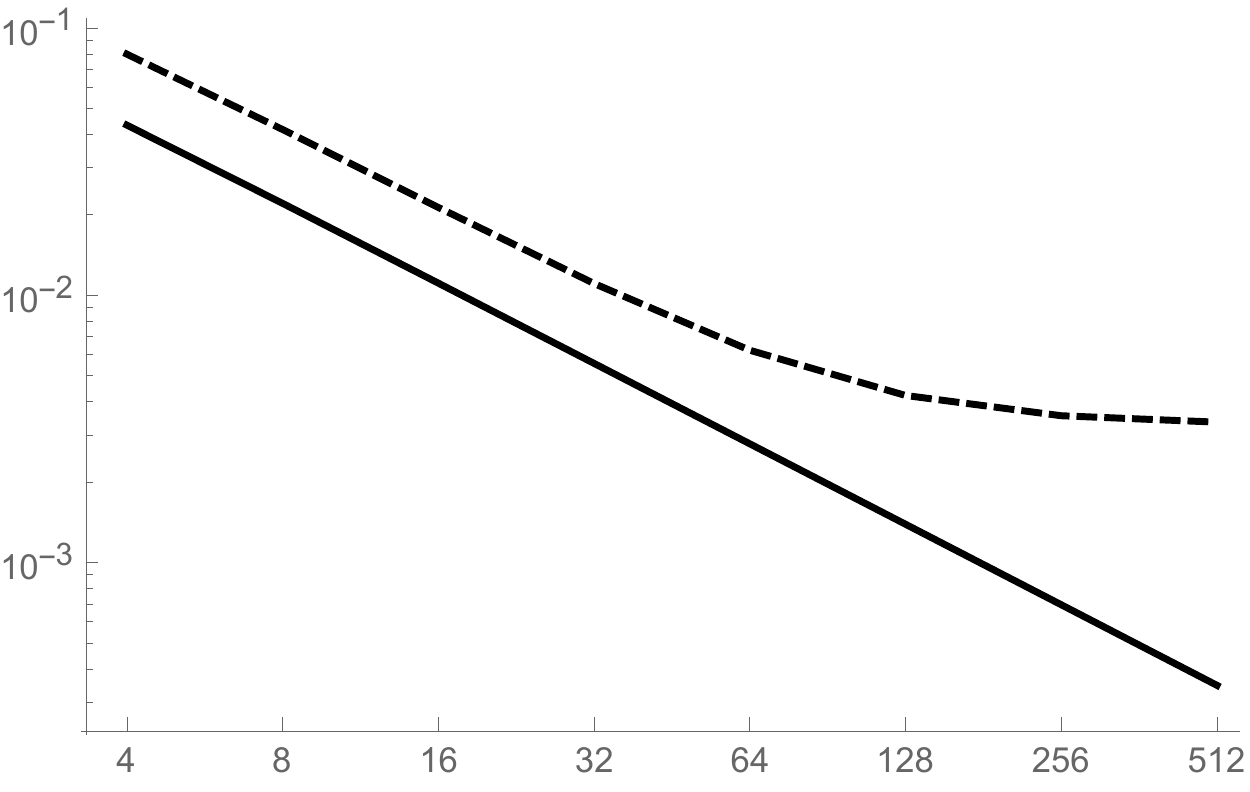}
\includegraphics[width=0.45\textwidth]{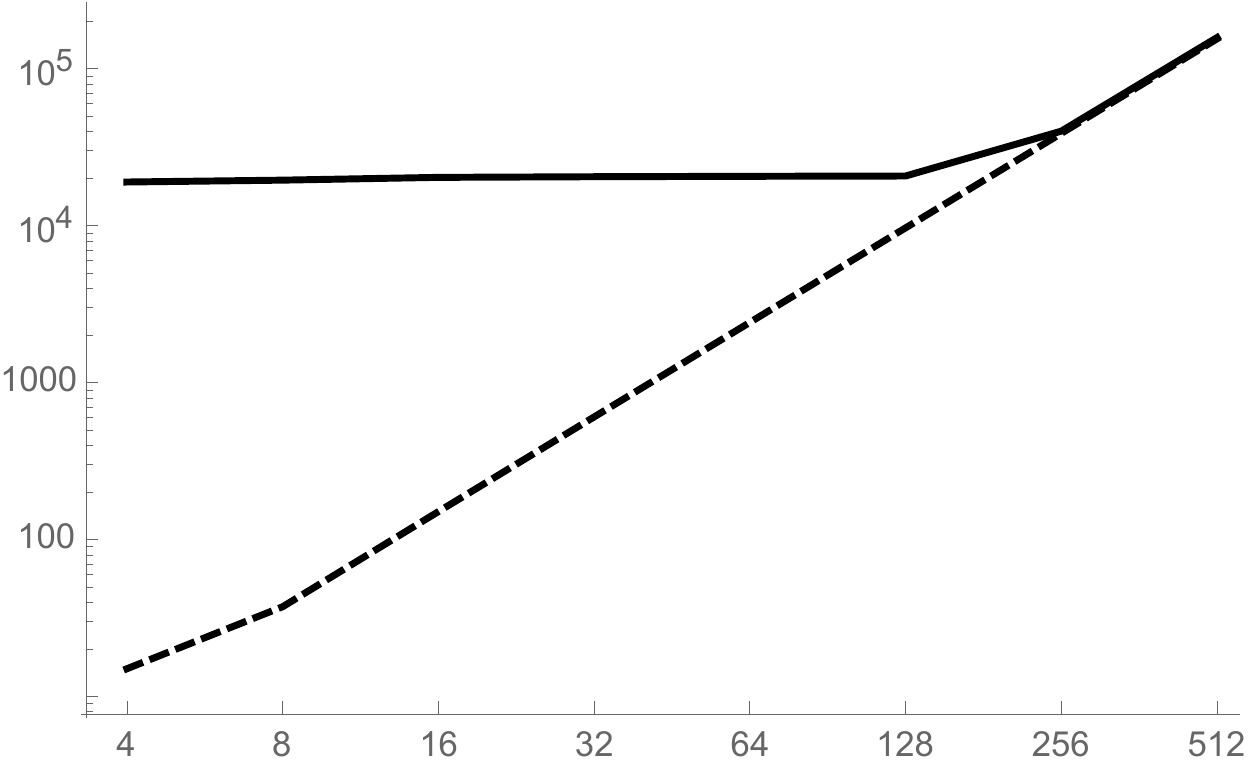}
\caption{\label{CurvyTriangleConvergence} Log-Log plots of $|u-\hat{u}|_{H^1(\Omega)}$ (left column)
  and $\kappa_2(A)$ (right column) with respect to the mesh parameter
  $r$ (horizontal axis), for Type 1 (dashed) and Type 2
  (solid) elements on Perturbed Triangle mesh families $a=1/4,1,4,16$
  (from top to bottom).}
\end{figure}
\end{example}

\begin{example}[L-Shaped Domain]\label{LShape}
For our final set of experiments, we again consider the problem
\begin{align*}
-\Delta u=1\mbox{ in }\Omega\quad,\quad u=0\mbox{ on }\partial\Omega~,
\end{align*}
but on the (rotated) L-shaped domain
$\Omega=(-1,1)\times (-1,1)\setminus [0,1]\times [-1,0]$ (see
Figure~\ref{LShapeFig}).  This solution, though not known explicitly,
is known to have a singularity at the origin, behaving asymptotically
like $|x|^{2/3}$ near the origin
(cf.~\cite{Wigley1964,MR3396210,Kozlov1997}).  In standard finite
element computations the efficient approximation of such a singular
solution would be achieved by targeted refinement of mesh cells toward
the singular point that is either guided by local error indicators
(computed a posteriori) or by specific knowledge of the local singular
behavior to determine an a priori mesh grading strategy.  A
head-to-head empirical comparison of these two types of refinement
strategies is provided in~\cite{liNM2014}.  Others have sought to
address the issue of singularities by augmenting standard polynomial
finite elements with (local) enrichment functions having the types of
singularities expected of the solution based on a priori knowledge
(cf. \cite{Fix1973}, \cite{Belytschko1999} XFEM, \cite{Stroboulis2000}
GFEM).   

Our approach for this problem is different.
We use the fact that, if a mesh cell $K$ has a
non-convex corner, then the local space $V_1(K)$ automatically
contains functions having the correct type of singularity for that
geometry.  Remark~\ref{InterpBestApprox} suggests improved
approximation power for interpolation of functions having the same
kind of singularities, and optimal order convergence was demonstrated
in~\cite{Anand2018} for interpolation 
error in $L^2(\Omega)$ of a harmonic function having an
$|x|^{2/3}$-type singularity on precisely the kinds of meshes shown
in Figure~\ref{LShapeFig}.  That work did not, however, consider
interpolation error or discretization error
in $H^1(\Omega)$  for such a problem.
As before, we use a parameter $r$ to
describe the meshes in this family.  
  The $r$th mesh in this family, $\cT_r$, consists of one L-shaped
  element,
  $K_L=(-1/3,1/3)\times (-1/3,1/3)\setminus[0,1/3]\times[1/3,0]$, and
  $24r^2$ congruent squares of size $(3r)^{-1}\times(3r)^{-1}$, see
  Figure~\ref{LShapeFig}.  We note that there are $r$ squares touching
  each of the short edges of $\partial K\setminus\partial \Omega$, and $\partial K_L$ has
  $6r+2$ vertices.  Although none of the edges in these meshes are
  curved, the optimal convergence rates enabled by the single L-shaped
  element illustrates how a result like~\eqref{InterpBestApprox_H1}
  might be used to prove what is empirically observed.  For this
  example, we provide such an analysis.

\begin{figure}
\begin{center}
\includegraphics[width=3.2in]{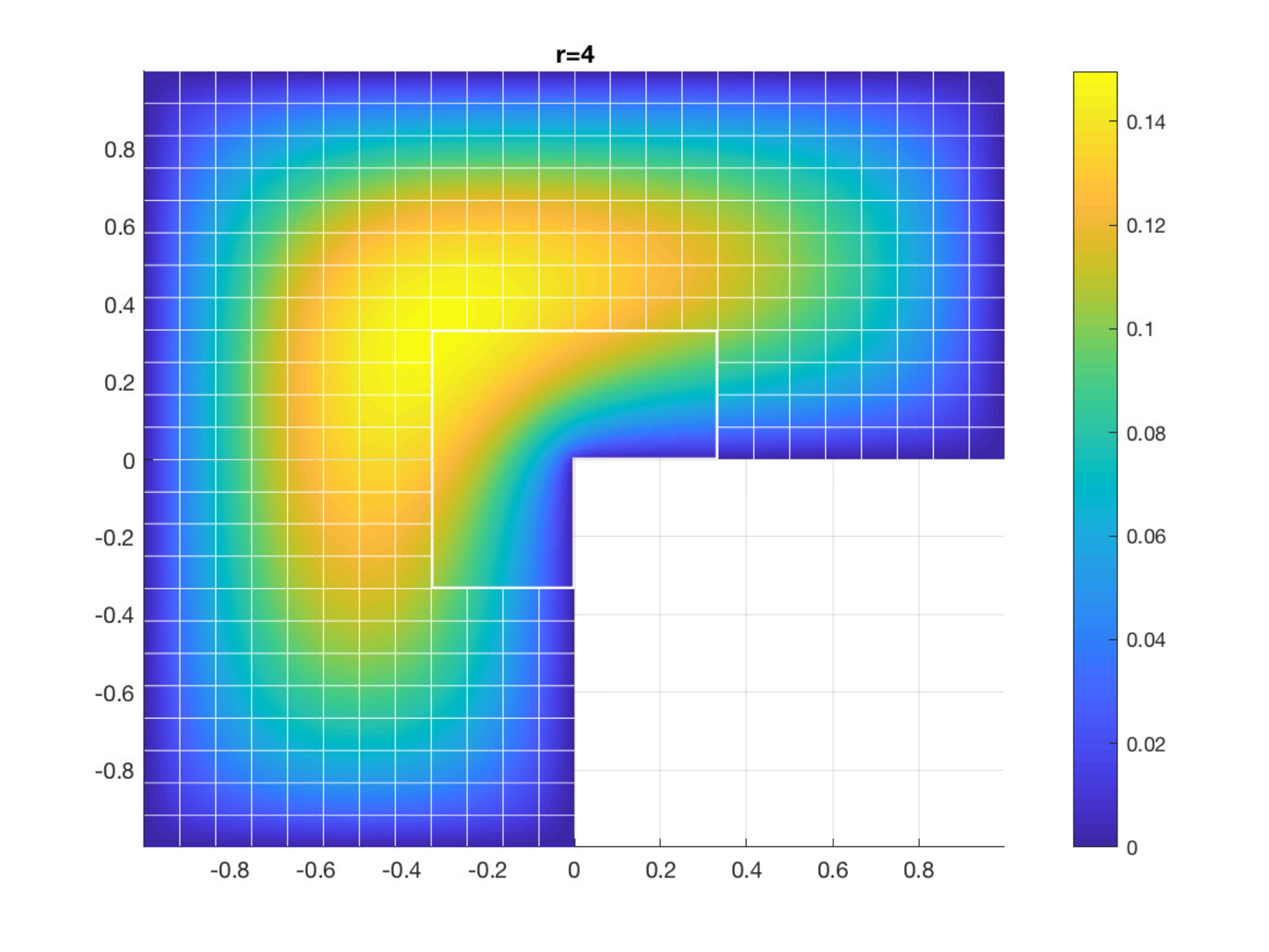}
\includegraphics[width=3.2in]{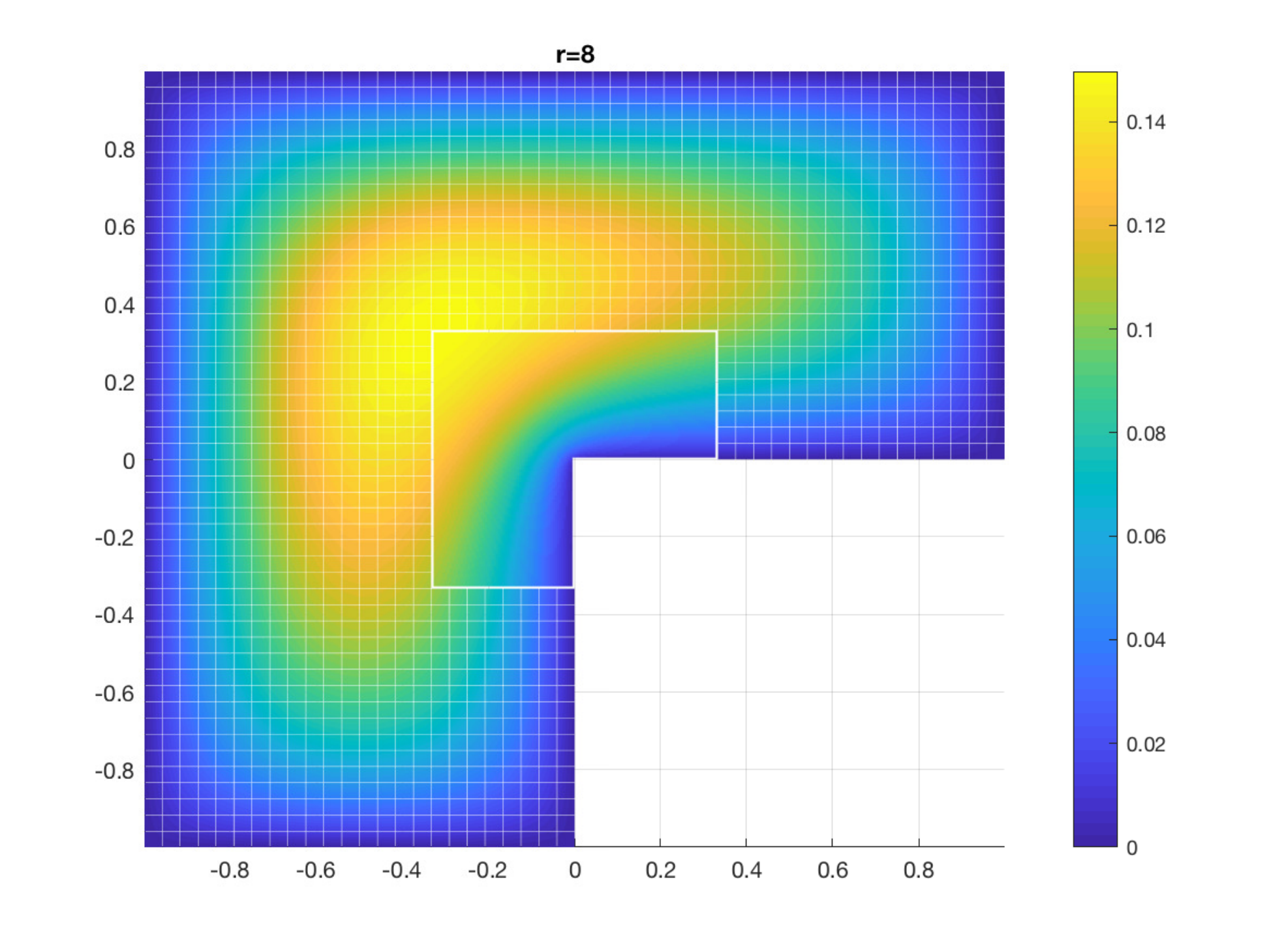}
\end{center}
\caption{\label{LShapeFig} The L-shaped domain
  together with the meshes and contour plots of the computed solutions
  $\hat{u}\in \widetilde{V}_1(\cT)$ corresponding to $r=4$ (left) and $r=8$.}
\end{figure}

As with Example~\ref{ThreeMeshesExample}, we use a highly accurate 
approximation of $|u|_{H^1(\Omega)}^2$, together with the identity
$|u-\hat{u}|_{H^1(\Omega)}^2=|u|_{H^1(\Omega)}^2-|\hat{u}|_{H^1(\Omega)}^2$.
In the previous example, we used a Fourier expansion to obtain our
approximation of $|u|_{H^1(\Omega)}^2$.  Here, we use the techniques 
developed in~\cite{Ovall2018}.  Letting $w=-|x|^2/4$, and recognizing
that $-\Delta w = 1$, we have $u=v+w$, where $\Delta v = 0$ in
$\Omega$ and $v=-w$ on $\partial\Omega$, and it follows that
\begin{align}
|u|_{H^1(\Omega)}^2&=|w|_{H^1(\Omega)}^2+|v|_{H^1(\Omega)}^2+2\int_{\Omega}\nabla
  v\cdot\nabla w\,ds\nonumber\\
&=\frac{1}{2}+\int_{\partial{\Omega}}\frac{\partial v}{\partial
  n}\,v\,ds+2\int_{\partial \Omega}\frac{\partial v}{\partial
  n}\,w\,ds\label{H1NormTrick}\\
&=\frac{1}{2}-\int_{\partial \Omega}\frac{\partial v}{\partial
  n}\,v\,ds
  \approx0.21407580269~.\nonumber
\end{align}
The integral $\int_{\partial \Omega}\frac{\partial v}{\partial
  n}\,v\,ds$ is approximated using the techniques from~\cite{Ovall2018}.

For the square elements $K$ in $\cT_r$, $V_1(K)$ consists of the
standard bilinear finite elements.  The element stiffness matrices for
these elements remain that same (up to symmetric permutation) for all
meshes,
\begin{align*}
A_K=\frac{1}{6}\begin{pmatrix}4&-1&-2&-1\\-1&4&-1&-2\\-2&-1&4&-1\\-1&-2&-1&4\end{pmatrix}~.
\end{align*}
The local space $V_1(K_L)$ for single L-shaped element $K_L$ changes
from mesh to mesh, so its element stiffness matrix $A_L$ must be
recomputed on each mesh.  The number of rows/columns of $A_L$ on
$\cT_r$ is $6r+2$.

Revisiting the interpolation identity~\eqref{InterpolationPythagorean}
in our present context, we have
\begin{align*}
|u-\cI_K u|_{H^1(K)}^2&= |u^K-\cI_K^K u|_{H^1(K)}^2+|u^{\partial
                        K}-\cI_K^{\partial K} u|_{H^1(K)}^2\\
&=|u^K|_{H^1(K)}^2+|u^{\partial K}-\cI_K^{\partial K} u|_{H^1(K)}^2~,
\end{align*}
because $V_1^{K}(K)=\{0\}$.  Since $-\Delta u^K=1$ in $K$ and $u^K=0$
on $\partial K$, we have 
\begin{align*}
|u^K|_{H^1(K)}^2=\int_K u^K\,dx\leq |K|^{1/2}\|u^K\|_{L^2(K)}\leq |K|^{1/2}h_K|u^K|_{H^1(K)}~.
\end{align*}
For all of the square elements $K$, which shrink as $r$ increases, the term
$|u^K|_{H^1(K)}$ is not problematic.  However, the L-shaped
element $K=K_L$ does not shrink as $r$ increases so the estimate
$|u^K|_{H^1(K)}\leq |K|^{1/2}h_K$ gives no guarantee of convergence at
all, much less at the optimal rate.  In fact, if we approximate $u$ by 
$\hat{u}\in V_1(\cT_r)$ for this family of meshes, we do not get
convergence in $H^1(\Omega)$!  

This issue is simple to fix, however, and the remedy we now describe
is suggestive of a more general principle that we aim to explore in
detail in subsequent work.  Because $-\Delta u^K = -\Delta u = 1$ on
$K=K_L$, we include the interior bubble function $\phi\in
V_2^{K}(K_L)$ satisfying $-\Delta \phi=1$ in $K_L$ and $\phi=0$ on
$\partial K_L$.  The necessary quantities associated with $\phi$
can be computed in the same manner as described~\eqref{H1NormTrick}
and its paragraph.  We take
$\widetilde{V}_1(K_L)=\mathrm{span}(V_1(K_L)\cup\{\phi\})$ and 
$\widetilde{V}_1(\cT_r)=\mathrm{span}(V_1(\cT_r)\cup\{\phi\})$.
We recall that $\int_{\Omega}\nabla \psi\cdot\nabla \phi\,dx=
\int_{K_L}\nabla \psi\cdot\nabla \phi\,dx=0$ for all $\psi\in
V_1(\cT_r)$, so adding this function does not increase the cost of 
assembling and solving the necessary linear system. 

With this enrichment by $\phi$, we have, on $K=K_L$,
\begin{align*}
|u-\cI_K u|_{H^1(K)}&= |u^{\partial K}-\cI_K^{\partial K} u|_{H^1(K)}~.
\end{align*}
Since $u^{\partial K}-\cI_K^{\partial K} u$ is harmonic on $K_L$,
Dirchlet's principle ensures that, on $K=K_L$,
\begin{align*}
|u^{\partial K}-\cI_K^{\partial K}
  u|_{H^1(K)}=\inf\{|\xi|_{H^1(\Omega)}:\,
\xi\in H^1(K_L)\mbox{ and }\xi = u^{\partial K}-\cI_K^{\partial K}
  u\mbox{ on }\partial K_L\}~.
\end{align*}
This can be estimated by a technique similar in spirit to that given
in \cite[Theorem 4.1]{GilletteRand2016}.  The argument involves
creating a (fictitious) sub-triangulation of $K_L$, taking $\phi$ to be
the piecewise linear interpolant of $u^{\partial K}$ (or $u$) on this
sub-triangulation, and using standard interpolation error estimates.
However, unlike the estimate in \cite[Theorem 4.1]{GilletteRand2016},
which assumes $H^2$ regularity of $u^{\partial K}$, we use
geometrically graded sub-triangulations, as pictured in
Figure~\ref{KLSubtriangulation}, and use estimates from
\cite{Apel1996,Li2009} %(see also \cite[Theorem 3.9]{liNM2014}).
to deduce that
$|u^{\partial K}-\cI_K^{\partial K} u|_{H^1(K)}\leq C_K r^{-1}$ on
$K=K_L$, where $C_K$ depends only on the mesh grading parameter and
the norm of $u^{\partial K}$ (or $u$) in an appropriate weighted
Sobolev space.  Combining this with the simple estimates from
Theorem~\ref{InterpolationError} for the square elements, and we see
that $|u-\cI u|_{H^1(\Omega)}\leq C r^{-1}$, which is the rate of convergence
observed in Table~\ref{LShapeTable}.  We emphasize that
sub-triangulations of $K_L$ are purely for the purpose of this
interpolation error estimate, and are not used at any point in the
actual computations.

\begin{figure}
\begin{center}
\includegraphics[width=2.7in]{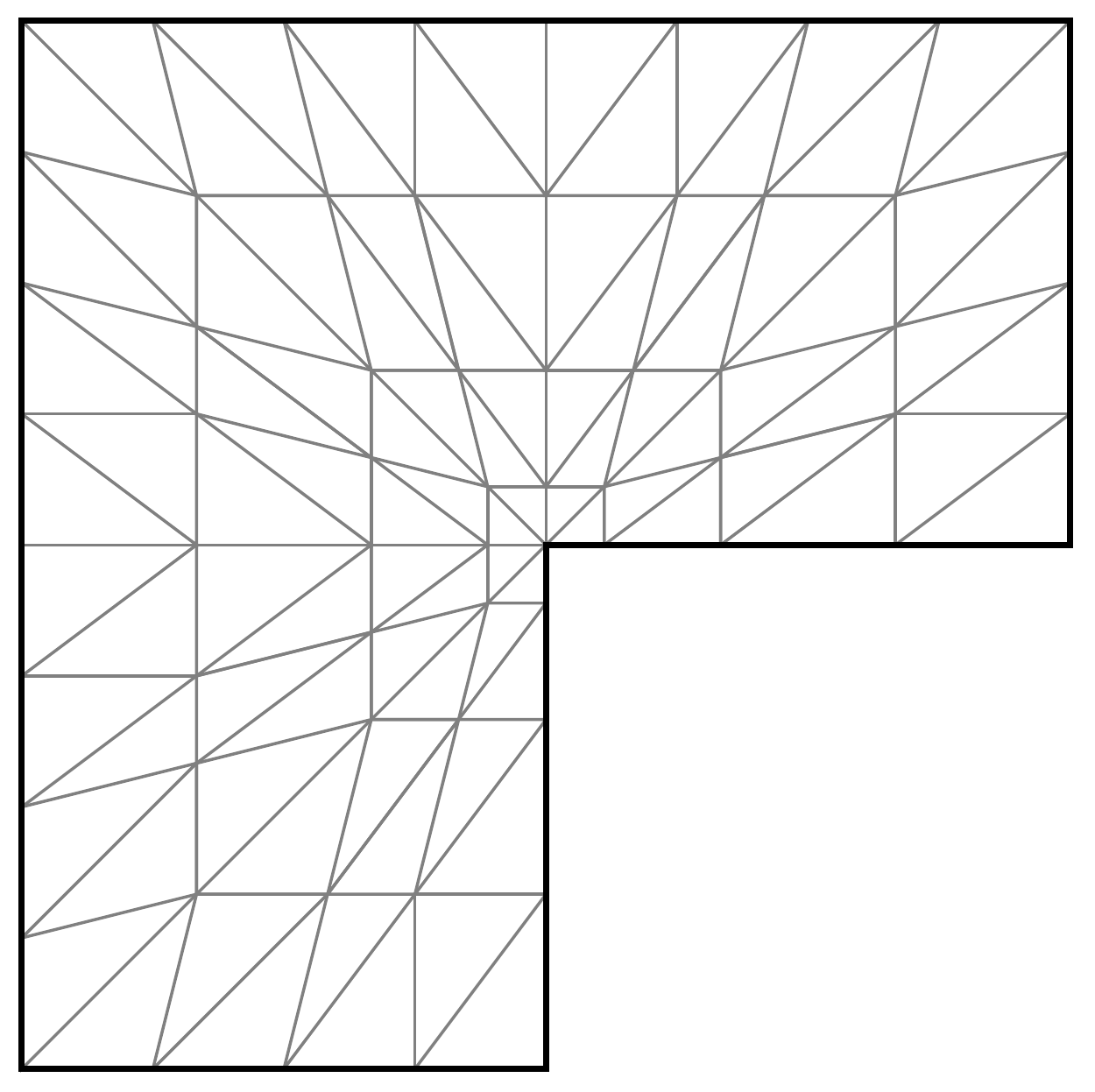}
\includegraphics[width=2.7in]{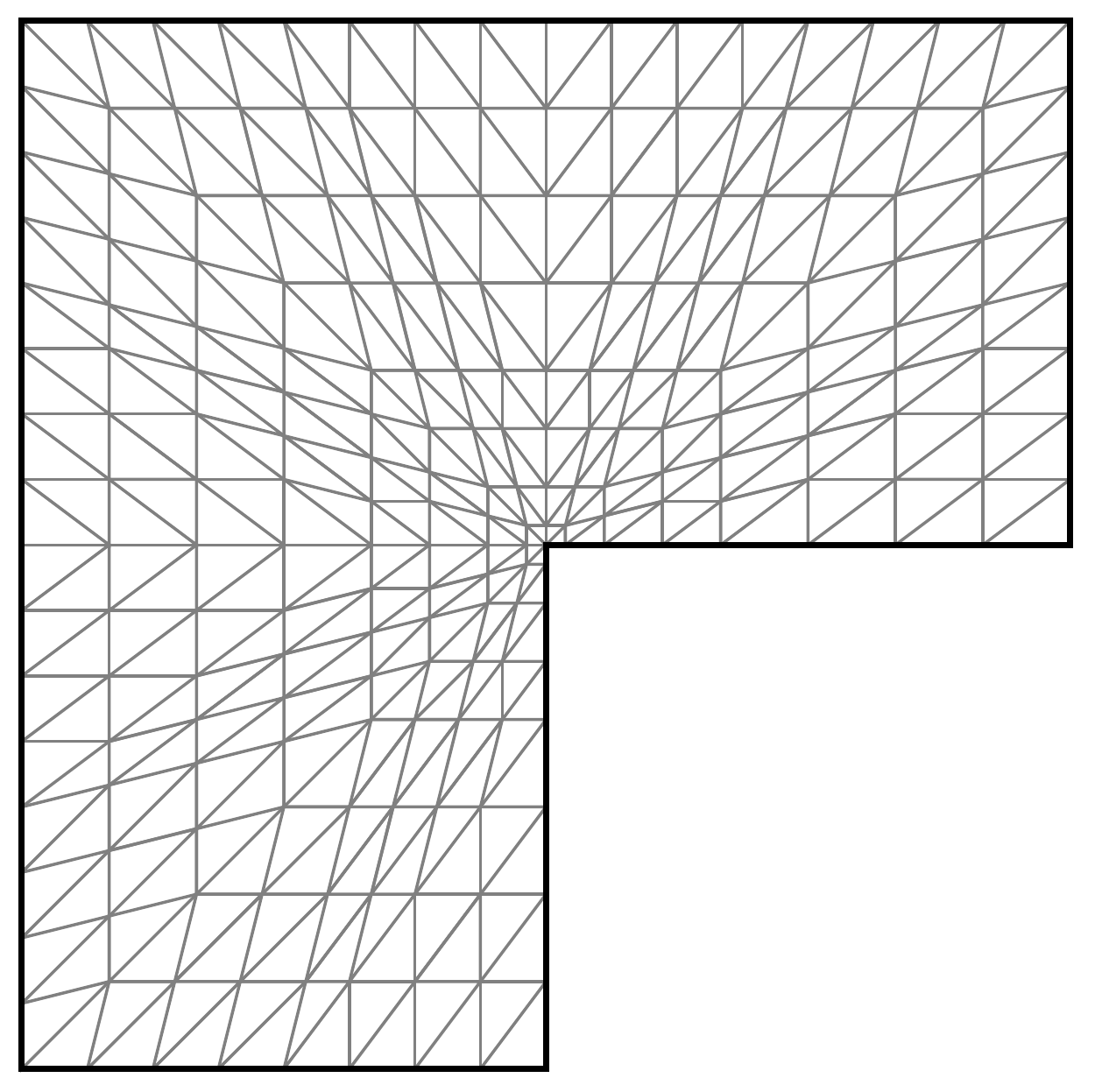}
\end{center}
\caption{\label{KLSubtriangulation} Sub-triangulations of
the L-shaped element $K_L$ that are geometrically graded toward the
origin, corresponding to $r=4$ (left) and $r=8$.  These
sub-triangulations are used merely as part of the argument to justify
the interpolation error estimate, and are not used in the computation
of $\hat{u}$.}
\end{figure}

The number of rows and columns (and non-zeros) of the global stiffness
matrix $A$ grows quadratically with $r$, and the number of rows and
columns of the (dense) submatrix $A_L$ corresponding to $V_1(K_L)$
grows linearly with $r$.  We comment briefly on the condition numbers
of these matrices, which are reported in Table~\ref{LShapeTable}.
As with standard (low-order) finite elements, the condition number of
$A$ grows linearly with $\dim(\widetilde{V}_1(\cT_r))$, or
quadratically with $r$.  In contrast, the growth of the condition
number of $A_L$ seems to be leveling off---it is certainly not growing
quadratically, or even linearly, with $r$.  These issues of
conditioning will be explored further in subsequent work.

\begin{table}
\caption{\label{LShapeTable} Discretization
  errors, $|u-\hat{u}|_{H^1(\Omega)}$, and error ratios for the
  sequence of meshes having one L-shaped cell of fixed size and
  increasingly fine square cells.  Also included are the spectral condition
  numbers of the global (sparse) stiffness matrix $A$ and the small
  (dense) stiffness matrix $A_{L}$ associated with the basis functions on the  L-shaped cell.}
\begin{center}
\begin{tabular}{|l|c|c|c|c|}\hline
$r$&$|u-\hat{u}|_{H^1(\Omega)}$&ratio&$\kappa_2(A)$&$\kappa_2(A_{L})$\\\hline
%     1 &   1.336e-01&         &12.91&5.784 \\
%     2 &   6.761e-02&2.016&37.68&7.570\\ 
%     4 &   3.373e-02&2.004&125.3&8.939\\ 
%     8  &  1.685e-02&2.002&463.4&9.820 \\
%    16 &  8.421e-03&2.001&1792 &10.38 \\\hline
    1&    	1.3629e-01&    &    		1.2908e+01&    5.7834e+00\\
    2&    	6.7610e-02&    2.0158&    	3.7679e+01&    7.5699e+00\\
    4&    	3.3734e-02&    2.0042&    	1.2527e+02&    8.9390e+00\\
    8&    	1.6855e-02&    2.0014&    	4.6339e+02&    9.8197e+00\\
    16&    	8.4305e-03&    1.9993&    	1.7919e+03&    1.0377e+01\\
    32&    	4.2289e-03&    1.9935&    	7.0585e+03&    1.0728e+01\\
    64&    	2.1468e-03&    1.9699&    	2.8029e+04&    1.2450e+01\\ \hline
\end{tabular}
\end{center}
\end{table}

% n=16
%        r             H1            E             R         GlobalCond     SubCond  
%    __________    __________    __________    __________    __________    __________
%    1.0000e+00    1.9550e-01    1.3629e-01    0.0000e+00    1.2908e+01    5.7834e+00
%    2.0000e+00    2.0950e-01    6.7610e-02    2.0158e+00    3.7679e+01    7.5699e+00
%    4.0000e+00    2.1294e-01    3.3734e-02    2.0042e+00    1.2527e+02    8.9390e+00
%    8.0000e+00    2.1379e-01    1.6855e-02    2.0014e+00    4.6339e+02    9.8197e+00
%    1.6000e+01    2.1400e-01    8.4305e-03    1.9993e+00    1.7919e+03    1.0377e+01
%    3.2000e+01    2.1406e-01    4.2289e-03    1.9935e+00    7.0585e+03    1.0728e+01
%    6.4000e+01    2.1407e-01    2.1468e-03    1.9699e+00    2.8029e+04    1.2450e+01

In~\cite[Example 4.4]{Anand2018} we compared interpolation errors in
$L^2(\Omega)$ for the harmonic function
$u=r^{2/3}\sin(2(\theta-\pi/2)/3)$ on
$\Omega=(-1,1)\times(-1,1)\setminus[0,1]\times[0,1]$, a rotated
version of the $\Omega$ used here, on three different families of
meshes.  As noted above, one of these families of meshes was the one
used here, and it led to optimal order convergence.  The two other
families yielded sub-optimal convergence at a theoretically predicted
rate.  One of these families of meshes consisted solely of congruent
squares, and the other family was the same except right near the
corner, where it had a single small L-shaped cell obtained by
merging three of these squares.  Although the local space on this
L-shaped element could approximate the singular function at the
optimal rate, the neighboring square elements, which got increasingly
closer to the singularity on finer meshes, could not, so the overall
convergence was spoiled.  This motivates our choice to keep the
L-shaped element of fixed size, as we have here.  The particular size
of this element is not crucial to the overall asymptotic behavior of
convergence.  In fact, we could have chosen $K_L=\Omega$ in this case,
and just solved the problem using integral equation techniques, as we
did above for~\eqref{H1NormTrick}.  The point of using the kinds of
meshes that we did here is to demonstrate that they can offer a
feasible alternative to more traditional refinement techniques.
\end{example}

\section{Conclusions}\label{Conclusions}
We have provided analysis and a practical low-order realization of a
novel finite element method on meshes consisting of quite general
curvilinear polygons.  Allowing for such curved elements introduces
both theoretical and computational challenges, including the proper
definition and treatment of polynomial spaces defined on curves,
determining an appropriate interpolation operator and obtaining
meaningful error estimates, and efficiently computing with the
implicitly-defined basis functions.  Concerning the first of these
challenges, we described and demonstrated simple methods for
constructing a spanning set for a polynomial space on an edge, and
then pairing it down to a basis.  Concerning the second challenge, we
proved local interpolation estimates in $L^2$ and $H^1$ for a
projection-based scheme, showing that interpolation in these spaces is
at least as good as interpolation in standard polynomial spaces on
typical element shapes (e.g. triangles and quadrilaterals).  The
optimal order convergence of finite element approximations of a
function having an unbounded gradient without employing small cells
near the singularity, as well as the analysis provided for that
specific example, indicates an even richer approximation theory that
will be explored in subsequent work.  In terms of practical
computations, we described a boundary integral approach that was very
recently developed with precisely these applications in mind.  The
numerical examples illustrated our convergence results on several
families of meshes whose mesh cells are far from being simple
perturbations of straight-edged polygons.  We also numerically
compared the approximation power of two types of harmonic bases on a
mesh consisting of triangles with a single perturbed edge.

As highlighted at the end of Section~\ref{Interpolation}, a better
understanding of how geometric features of elements and choice of $p$
affect constants appearing in the interpolation analysis for $V_p(K)$
is needed.  Additionally, an accounting of errors made in the
approximation of quantities required for the formation of element
stiffness matrices, and those subsequently arising from quadratures,
should be taken into account as part of a more complete analysis of
the method.  Since the experiments in this work really only involved
harmonic basis functions, all quadratures were performed on the
boundaries of elements, but higher-order elements will require
volumetric quadratures as well, and the development of efficient and
robust volumetric quadratures in our context is another topic for
further investigation.  The method should also be tested on PDEs
modeling more complex phenomena, and problems in which there are
curved (and moving) interfaces between materials are of particular
interest in this regard.  Extending this approach to 3D problems in
which general curved cells are permitted presents both theoretical and
practical/computational challenges beyond the obvious analogues
discussed above, and we aim to address them in future work.  Among
these is a definition of $V_p(K)$ that leads to a conforming space
$V_p(\cT)$ without making the dimension of $V_p(K)$ much larger than
is necessary to achieve optimal approximation properties.  A second
challenge is the efficient and accurate evaluation of quantities that
are needed to form the local finite element linear systems; the
approach outlined in Section~\ref{Computation} is inherently 2D, but
there are integral equation approaches that may prove beneficial in
our setting.

%%%%%%%%%%%%%%%%
% begin bibliography
%%%%%%%%%%%%%%%%
%\bibliography{titles}
%\bibliographystyle{siam}
%\bibliographystyle{abbrv}
\def\cprime{$'$}

%%%%%%%%%%%%%%%%
% end bibliography
%%%%%%%%%%%%%%%%

\end{document}